\documentclass{amsart}

\usepackage{amsmath, amssymb, amsfonts}
\usepackage{graphicx}

\newcommand{\eps}{\varepsilon}

\newcommand{\epsilonr}{\epsilon_r}
\newcommand{\epsiloni}{\epsilon_i}

\newcommand{\hepsilonr}{\hat{\epsilon}_r}
\newcommand{\hepsiloni}{\hat{\epsilon}_i}

\newcommand{\Rl}{\mathbb{R}}

\newcommand{\Nl}{\mathbb{N}}
\newcommand{\Ir}{\mathbb{Z}}
\newcommand{\Ts}{\mathbb{T}}

\newcommand{\sgn}{\operatorname{sgn}}

\renewcommand{\Pr}{\operatorname{\mathbf{P}}}
\newcommand{\Qr}{\operatorname{\mathbf{Q}}}
\newtheorem{theorem}{Theorem}

\newtheorem{definition}[theorem]{Definition}
\newtheorem{lemma}[theorem]{Lemma}

\title{Nonlinear Surface Plasmons}

\author{Ryan G. Halabi}
\address{ Department of Mathematics, University of California at Davis}
\email{rghalabi@math.ucdavis.edu}
\author{John K. Hunter}
\address{ Department of Mathematics, University of California at Davis}
\thanks{The second author was partially supported by the NSF under grant number  DMS-1312342.}
\email{jkhunter@ucdavis.edu}

\date{October 27, 2015}

\begin{document}

\begin{abstract}
We derive an asymptotic equation for quasi-static, nonlinear surface plasmons propagating on a planar interface between isotropic media. The plasmons are nondispersive with a constant linearized frequency that is independent of their wavenumber. The spatial profile of a weakly nonlinear plasmon satisfies a nonlocal, cubically nonlinear evolution equation that couples its left-moving and right-moving Fourier components. We prove short-time existence of smooth solutions of the asymptotic equation and describe its Hamiltonian structure. We also prove global existence of weak solutions of a unidirectional reduction of the asymptotic equation.  Numerical solutions show that nonlinear effects can lead to the strong spatial focusing of plasmons. Solutions of the unidirectional equation appear to remain smooth when they focus, but it is unclear whether or not focusing can lead to singularity formation in solutions of the bidirectional equation.
\end{abstract}

\maketitle

\section{Introduction}
Surface plasmon polaritons (or SPPs) are electromagnetic surface waves that propagate on an interface between a dielectric and a conductor and decay
exponentially away from the interface; for example, optical SPPs propagate on an interface between air and gold.
Maradudin et.~al.\ \cite{plasmonics} give an overview of SPPs and their applications in plasmonics. Kauranen and Zayats \cite{zayats} review nonlinear aspects
of plasmonics.

We model SPPs by the classical macroscopic Maxwell equations, and
consider the basic case of SPPs that propagate along a planar interface separating an isotropic dielectric in the upper half-space from an
isotropic conductor in the lower half-space (see Figure~\ref{fig:spp}).

\begin{figure}[t]
\centering
\includegraphics[width=2.5in]{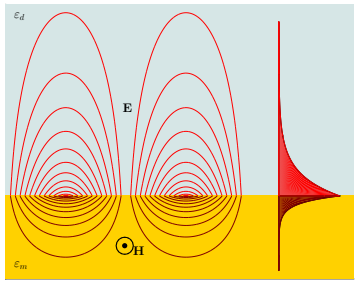}
\caption{An SPP on the interface between a dielectric and a metal. The electric field lines are shown in the $(x,y)$-plane and the field decays exponentially away from the interface. The magnetic field is pointing inward or outward in the $z$-direction.}
\label{fig:spp}
\end{figure}

Figure~\ref{fig:disp} shows a typical linearized dispersion relation for such SPPs.
The phase speed of long-wavelength SPPs approaches a constant speed $c_0$, and the frequency of short-wavelength SPPs approaches a constant
frequency $\omega_0$. This limiting frequency  is determined by the condition that
\begin{equation*}
\hat{\epsilon}^+(\omega_0) + \hat{\epsilon}^-(\omega_0) = 0
\end{equation*}
where $\hat{\epsilon}^+ > 0$ is the permittivity of the dielectric and $\hat{\epsilon}^- < 0$ is the permittivity of the conductor, expressed as
functions of the frequency. In this short-wave limit, the electromagnetic field is approximately quasi-static, and we will refer to the corresponding oscillations
on the interface as surface plasmons (SPs) for short.

\begin{figure}
\centering
\includegraphics[width=3.5in]{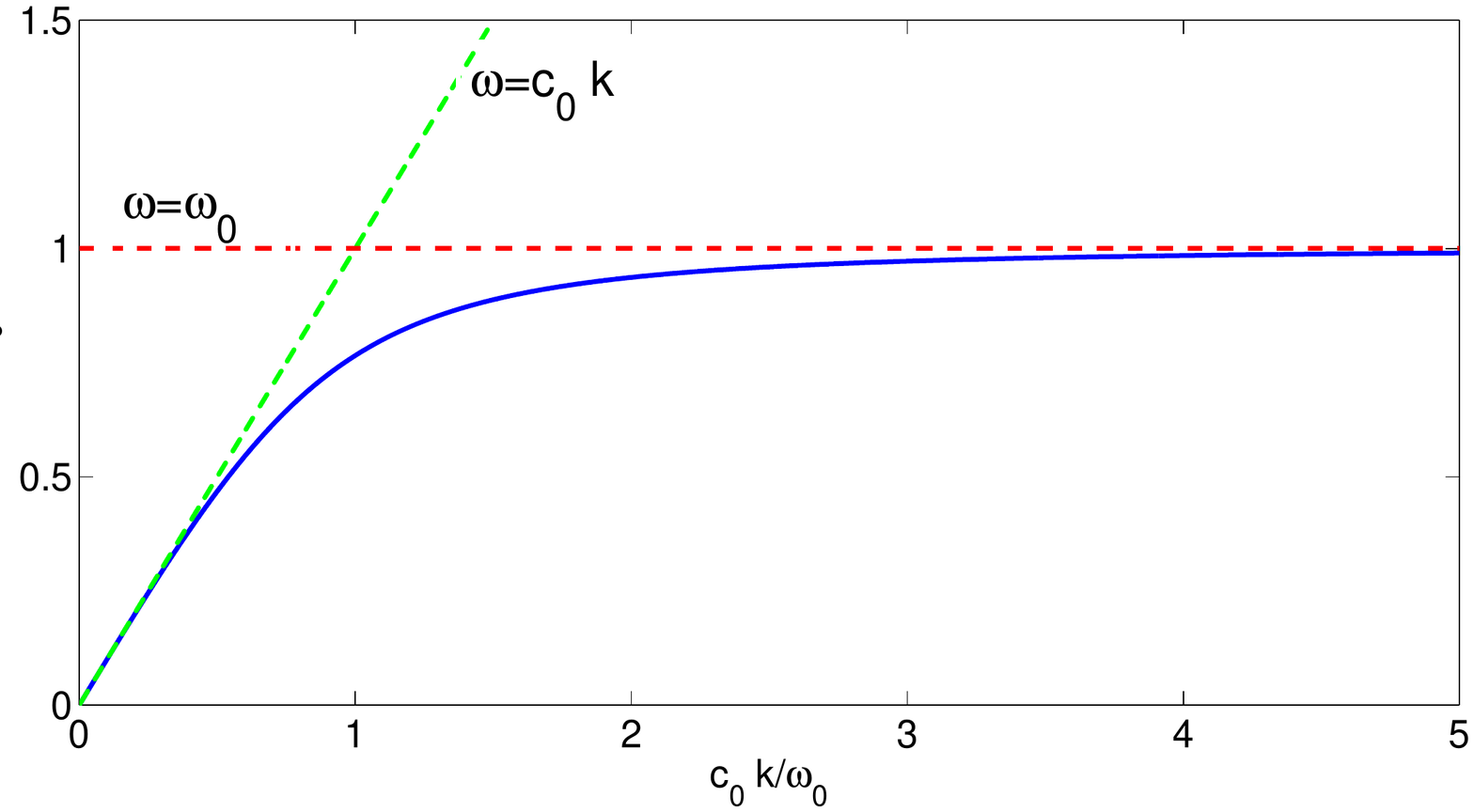}
\caption{The linearized dispersion relation \eqref{drude_disp} for SPPs on the interface between a vacuum and a Drude metal.}
\label{fig:disp}
\end{figure}

Nonlinear SPs are interesting to study because their linearized frequency is independent of their wavenumber.
As a result, they are nondispersive with zero group velocity, and weakly nonlinear SPs are subject to a large family of cubically nonlinear, four-wave resonant interactions
between wavenumbers $\{k_1,k_2,k_3,k_4\}$ such that
\[
k_1 + k_2 = k_3 + k_4.
\]
The corresponding resonance condition for the frequencies, $\omega_0 + \omega_0 = \omega_0 +\omega_0$, is satisfied automatically.

The nonlinear self-interaction of an SP therefore leads to a wave with a narrow frequency spectrum and a wide wavenumber spectrum, meaning that nonlinearity distorts the spatial profile of the wave. This behavior is qualitatively different from NLS-type descriptions of dispersive waves in nonlinear optics,
where the spatial profile consists of a slowly modulated, harmonic wavetrain \cite{newell}. Thus, the limit considered here is useful in understanding the
spatial dynamics of short-wave optical pulses in the case of surface plasmons.

We remark that if the spatial dispersion of the optical media is negligible, as we assume here, then the behavior of the SP depends only on the response of the media
at the frequency $\omega_0$, so we do not need to make any specific assumptions about the frequency-dependence of their linear permittivity or nonlinear susceptibility.

A related example of constant-frequency surface waves on a vorticity discontinuity in an inviscid, incompressible fluid is analyzed in \cite{biello},
which gives further discussion of constant-frequency waves.

Our asymptotic solution for the tangential electric field in the $x$-direction of a weakly nonlinear SP on the interface $z=0$ has the form
\begin{equation}
E_\parallel(x,0,t) = \eps \left[A(x,\eps^2t) e^{-i\omega_0 t} + A^*(x,\eps^2t) e^{i\omega_0 t}\right] + O(\eps^3),
\label{epar}
\end{equation}
as $\eps \to 0$ with $t=O(\eps^{-2})$, where $A(x,\tau)$ is complex-valued amplitude.
This solution consists of an oscillation of frequency $\omega_0$ whose spatial profile
evolves slowly in time.

To write the evolution equation for $A$ in a convenient form, we introduce the projections $\Pr$, $\Qr$ onto positive and negative wave numbers, and define
\begin{align}
\begin{split}
u (x,\tau) &= \Pr[A](x,\tau)
=  \int_0^\infty \hat{A}(k,\tau) e^{ikx}\, dk,
\\
v(x,\tau) &= \Qr[A](x,\tau)
=  \int_{-\infty}^0 \hat{A}(k,\tau) e^{ikx}\, dk,
\end{split}
\label{def_uv}
\end{align}
where
\[
A(x,\tau) = \int_{-\infty}^\infty \hat{A}(k,\tau) e^{ikx}\, dk.
\]
For spatially periodic solutions, the Fourier integrals are replaced by Fourier series. Then
\begin{align*}
&E_\parallel = \eps\left[u e^{-i\omega_0t} + u^* e^{i\omega_0t}\right] + \eps\left[v e^{-i\omega_0t} + v^* e^{i\omega_0t}\right]
+ O(\eps^3),
\end{align*}
where $u$ is the amplitude of the right-moving, positive-frequency waves and
$v$ is the amplitude of the left-moving, positive-frequency waves.

In the absence of damping and dispersion, the complex-valued amplitudes $u(x,\tau)$, $v(x,\tau)$ satisfy the following cubically nonlinear, nonlocal asymptotic equations
\begin{align}
\begin{split}
u_{\tau} &= i\partial_x \Pr\left[\alpha u^*\partial_x^{-1}(u^2) + \beta v \partial_x^{-1}(uv^*)\right], \qquad \Pr[u] = u,
\\
v_{\tau} &= i\partial_x \Qr\left[\alpha v^* \partial_x^{-1}(v^2) + \beta u \partial_x^{-1}(u^*v)\right],\qquad\Qr[v]=v,
\end{split}
\label{uv_eq}
\end{align}
where the coefficients $\alpha, \beta\in \Rl$ are proportional to sums of the nonlinear susceptibilities of the media on either side of the interface,
and the inverse derivative is defined spectrally by
\[
\partial_x^{-1}(e^{ikx}) = \frac{1}{ik}e^{ikx}.
\]
Here, we assume that the Fourier transforms of $u$, $v$ vanish to sufficient order at $k=0$ in the spatial case $x\in \Rl$,
or that $u$, $v$ have zero mean in the spatially periodic case $x\in \Ts$. As we show in Section~\ref{sec:short_time}, the spatially
periodic form of \eqref{uv_eq} is an ODE in the $L^2$-Sobolev space $H^s(\Ts)$ for $s > 1/2$.

The inclusion of weak short-wave damping  and dispersion in the asymptotic expansion leads to a more general equation \eqref{uv_eq_disp}, which consists of \eqref{uv_eq}
with additional lower-order, linear terms. The spectral form of this equation
for $\hat{A} = \hat{u} + \hat{v}$ is given in \eqref{spec_eq}.

A basic feature of \eqref{uv_eq} is that the actions of the left and right moving waves, given in \eqref{uv_action},
are separately conserved. In quantum-mechanical terms, this means that the numbers of left and right moving plasmons are conserved.
In billiard-ball terms, the collision of two right-moving plasmons produces two right-moving plasmons; recoil into a low-momentum, left-moving plasmon and a high-momentum, right-moving plasmon is not allowed.

This property is explained by the fact that surface plasmons have helicity \cite{helicity}, even though they are plane-polarized, and the helicities
of left and right moving surface plasmons have opposite signs. As a result, conservation of angular momentum implies that nonlinear interactions cannot
convert right-moving plasmons into left-moving plasmons, or visa-versa.

In particular, setting $v=0$ in \eqref{uv_eq} and normalizing $\alpha=1$, we get an equation for unidirectional surface plasmons
\begin{equation}
u_{\tau} = i\partial_x \Pr\left[u^*\partial_x^{-1}(u^2)\right],\qquad \Pr[u] = u.
\label{u_eq}
\end{equation}
As we discuss further in Section~\ref{sec:ham},
this equation is related to the completely integrable Szeg\"o equation,
\begin{equation}
i u_\tau = \Pr\left[|u|^2 u\right],\qquad \Pr[u] = u,
\label{szego_eq}
\end{equation}
which was introduced by G\'erard and Grellier \cite{szego} as a model for totally nondispersive evolution equations. In our context,
the projection operator arises naturally from the condition that nonlinear interactions between positive-wavenumber
components do not generate negative-wavenumber components.

In the rest of this paper, we describe the linearized equations for SPPs and derive the asymptotic equations for nonlinear SPs.
We prove a short-time existence result for the asymptotic equation \eqref{uv_eq}, describe its Hamiltonian structure, and prove the global existence
of weak solutions for the unidirectional equation \eqref{u_eq}. We also present some numerical solutions, which show that
nonlinear effects can lead to strong spatial focusing of the SPs. Solutions of the unidirectional equation \eqref{u_eq} appear to remain smooth through the focusing, but it is unclear whether or not singularities form in solutions of the system \eqref{uv_eq}.

\section{Maxwell's equations}

The macroscopic Maxwell equations for electric fields $\mathbf{E}$, $\mathbf{D}$ and magnetic fields $\mathbf{B}$, $\mathbf{H}$ are
\begin{align*}
&\nabla \cdot \mathbf{B} = 0 , \qquad \nabla \times \mathbf{E} + \mathbf{B}_t = 0,
\\
&\nabla \cdot \mathbf{D} = 0, \qquad \nabla \times \mathbf{H} -\mathbf{D}_t = 0,
\end{align*}
where we assume that there are no external charges or currents.

We consider nonmagnetic, isotropic media without spatial dispersion, in which case we have the constitutive relations
\begin{align}
\begin{split}
&\mathbf{D}(\mathbf{x},t) =
 \int_{-\infty}^t \epsilon(t-t') \mathbf{E}(\mathbf{x},t') \, dt'
 \\
 &+ \int_{-\infty}^t \chi(t-t_1, t-t_2, t- t_3) [\mathbf{E}(\mathbf{x},t_1) \cdot
 \mathbf{E}(\mathbf{x},t_2)] \mathbf{E}(\mathbf{x},t_3) \, dt_1 dt_2 dt_3 + O(|\mathbf{E}|^5),
 \\
&\mathbf{B}(\mathbf{x},t) = \mu \mathbf{H}(\mathbf{x},t),
\end{split}
\label{non_dim_con}
\end{align}
where $\epsilon$ is the permittivity, $\chi$ is a third-order nonlinear susceptibility,
and the magnetic permeability $\mu$ is a constant.

We suppose that the interface between the dielectric and the conductor is located at $z=0$, with
\begin{equation}
\epsilon, \chi = \begin{cases}\epsilon^+, \chi^+ &\mbox{if $z>0$,}\\ \epsilon^-, \chi^- &\mbox{if $z<0$,}\end{cases}
\label{eps_pm}
\end{equation}
and assume that both media have the same magnetic permeability $\mu$. The jump conditions across $z=0$ are
\begin{align}
&[\mathbf{n}\cdot\mathbf{B}] = 0 , \qquad [\mathbf{n}\times\mathbf{E}] = 0,
\qquad [\mathbf{n}\cdot\mathbf{D}] = 0 , \qquad [\mathbf{n}\times\mathbf{H}] = 0,
\label{jump_con}
\end{align}
where $\mathbf{n} = (0,0,1)^T$ is the unit normal,
$[F] = F^+ - F^-$ denotes the jump in $F$ across the interface, and we assume that there are no surface charges or currents.
We further require that the fields approach zero as $|z|\to\infty$.

We consider transverse-magnetic SPPs that propagate in the $x$-direction and decay in the $z$-direction,
with $\mathbf{E}$ in the $(x,z)$-plane and $\mathbf{B}$ in the $y$-direction.
Solution of the linearized equations for Fourier-Laplace modes gives an electric field of the form \cite{plasmonics}
\[
\mathbf{E}(x,z,t) = \begin{cases} \hat{\mathbf{\Phi}}^+ e^{ikx - \beta^+ |k| z - i\omega t} &\mbox{in $z>0$,}\\
\hat{\mathbf{\Phi}}^- e^{ikx + \beta^-|k| z - i\omega t} &\mbox{in $z<0$,}\end{cases}
\]
where $\beta^\pm$ are positive constants and $\hat{\mathbf{\Phi}}^\pm$ are constant vectors in the $(x,z)$-plane. The frequency $\omega(k)$ satisfies the
dispersion relation
\begin{equation}
k^2 = \mu\omega^2\left[\frac{\hat{\epsilon}^+(\omega) \hat{\epsilon}^-(\omega)}{\hat{\epsilon}^+(\omega) + \hat{\epsilon}^-(\omega)}\right],
\label{gen_disp}
\end{equation}
where $\hat{\epsilon}(\omega)$ denotes the Fourier transform of $\epsilon(t)$.

As a typical example, the permitivities for an interface between a vacuum and a loss-less Drude metal are given by
\[
\hat{\epsilon}_+ = \epsilon_0,
\qquad
\hat{\epsilon}_-(\omega) = \epsilon_0\left(1 - \frac{\omega_p^2}{\omega^2}\right),
\]
where $\omega_p$ is the plasma frequency of the metal. In that case, the dispersion relation \eqref{gen_disp} becomes
\begin{equation}
\omega^2 =  \omega_0^2 + c_0^2 k^2 - \sqrt{\omega_0^4 + c_0^4 k^4},
\label{drude_disp}
\end{equation}
where the limiting frequency $\omega_0$ and speed $c_0$ are given by
\[
\omega_0 = \frac{\omega_p}{\sqrt{2}},\qquad c_0 = \frac{1}{\sqrt{\epsilon_0\mu}}.
\]
This dispersion relation is plotted in Figure~\ref{fig:disp}.

For example, one estimate of the plasma frequency for the Drude model of gold, cited in \cite{west}, is $\omega_p \approx 1.35 \times 10^{16}$ $\text{rad}\, \text{s}^{-1}$. The corresponding limiting frequency of an SP on a vacuum-gold interface is $\omega_0 \approx 9.5 \times 10^{15}$ $\text{rad}\, \text{s}^{-1}$, which lies in the near ultraviolet.
The wavenumber $k_0 = \omega_0/c_0$ is given by $k_0 \approx 3.2 \times 10^{7}$ $\text{m}^{-1}$, and the frequency of an SP is close to $\omega_0$
when $k \gg k_0$, corresponding to wavelengths $\lambda \ll 200$ nm. Ultraviolet frequencies are outside the range of usual applications of SPs,
but, with the inclusion of weak dispersive effects, our asymptotic solution applies to SPs with somewhat smaller frequencies. Moreover, different parameter values are relevant for other plasmonic materials than gold.

If damping effects are included in the free-electron Drude model for the metal, then one gets a complex permittivity
\[
\hat{\epsilon}(\omega) = \epsilon_0\left(1- \frac{\omega_p^2}{\omega^2 + i \gamma\omega}\right)
=   \epsilon_0\left(1- \frac{\omega_p^2}{\omega^2} + \frac{i\gamma\omega_p^2}{\omega^3}
+ O(\gamma^2)\right).
\]
Typical values of $\gamma/\omega_0$ are of the order $10^{-2}$, corresponding to relatively weak damping of an SP over the time-scale of its
period.

\section{Asymptotic expansion}

In order to carry out an asymptotic expansion for quasi-static SPs, we first non-dimensionalize Maxwell's equations. Let $\omega_0$,
$c_0$, $\epsilon_0$, and $\mu_0$  be a characteristic frequency, wave speed, permittivity, and permeability for the SP,
with $c_0^2 = (\epsilon_0\mu_0)^{-1}$, and
let $E_0$ denote a characteristic electric field strength at which the
response of the media is nonlinear. We define corresponding wavenumber and magnetic field parameters by
\[
 k_0 = \frac{\omega_0}{c_0},\qquad
B_0 = \frac{E_0}{c_0}.
\]

The quasi-static limit applies to wavenumbers $k \gg k_0$. If $\lambda$ is a characteristic length-scale for variations
of the electric field, then we assume that
\[
\eps = k_0 \lambda  \ll 1
\]
is a small parameter.
If $E_*$ is a typical magnitude of the electric field strength in the SP, then we also assume that
\[
 \delta= \frac{E_*}{E_0} \ll 1.
\]
In the expansion below, we take $\delta = \epsilon$, which leads to a balance between weak nonlinearity and weak short-wave dispersion.
The nondispersive equations \eqref{uv_eq} apply in the limit $\eps \ll \delta \ll 1$.

We define dimensionless variables, written with a tilde,  by
\begin{align*}
&\tilde{\mathbf{x}} = \frac{\mathbf{x}}{\lambda}, \qquad \tilde{t} = \omega_0 t,
\qquad \tilde{\mathbf{E}} = \frac{\mathbf{E}}{E_0},\qquad \tilde{\mathbf{D}} = \frac{\mathbf{D}}{\epsilon_0E_0},\qquad
\tilde{\mathbf{B}} = \frac{\mathbf{B}}{B_0},\qquad\tilde{\mathbf{H}} = \frac{\mu_0\mathbf{H}}{B_0}.
\end{align*}
The corresponding dimensionless permittivity,  nonlinear susceptibility, and permeability are
\[
\tilde{\epsilon} = \frac{\epsilon}{\epsilon_0 \omega_0},\qquad \tilde{\chi} = \frac{E_0^2\chi}{\epsilon_0\omega_0^3},\qquad
\tilde{\mu} = \frac{\mu}{\mu_0}.
\]

After dropping the tildes, we get the nondimensionalized Maxwell equations
\begin{align}
\begin{split}
&\nabla \cdot \mathbf{B} = 0, \qquad \nabla \times \mathbf{E} + \eps \mathbf{B}_t = 0,
\\
&\nabla \cdot \mathbf{D} = 0, \qquad
\nabla \times \mathbf{H} - \eps   \mathbf{D}_t = 0,
\end{split}
\label{non_dim_max}
\end{align}
with the constitutive relations \eqref{non_dim_con}, where
$\epsilon=\epsilon^\pm$ and $\chi = \chi^\pm$  are the nondimensionalized permitivities and
nonlinear susceptibilities in $z > 0$ and $z<0$, and $\mu$ is the nondimensionalized
permeability.

The simplest form of the equations, in which we neglect dispersion and include only nonlinearity, arises when we take $\eps=0$ in \eqref{non_dim_max}.
In that case, $\mathbf{H}=\mathbf{B} = 0$ and $\mathbf{D}$, $\mathbf{E}$ satisfy the purely electrostatic equations
\[
\nabla \times \mathbf{E}= 0,
\qquad
\nabla \cdot \mathbf{D} = 0,
\]
with the time-dependent constitutive relation in \eqref{non_dim_con}.

We assume that, in the frequency range of interest, the permittivity has an expansion
\begin{equation}
\epsilon = \epsilonr + i\eps^2\epsiloni,
\label{perm_exp}
\end{equation}
where $\epsilonr$ is the leading-order real part of $\epsilon$ and $\eps^2\epsiloni$ is a small imaginary part that describes
weak damping of the SP. We also assume that the susceptibility $\chi$ is real-valued to leading order in $\eps$.

We denote the Fourier transforms of $\epsilonr$, $\epsiloni$, and $\chi$
by $\epsilonr$, $\hepsiloni$, and $\hat{\chi}$, where
\begin{align*}
&\hepsilonr(\omega) = \int \epsilon_r(t)e^{i\omega t}\, dt,\qquad\hepsiloni(\omega) = \int \epsilon_i(t)e^{i\omega t}\, dt,
\\
&\hat{\chi}(\omega_1,\omega_2,\omega_3) = \int \chi(t_1,t_2,t_3) e^{i\omega_1 t_1+i\omega_2 t_2+i\omega_3 t_3}\, dt_1 dt_2 dt_3.
\end{align*}
These transforms have the symmetry properties
\begin{align}
\begin{split}
&\hepsilonr(-\omega) = \hepsilonr(\omega),\qquad \hepsiloni(-\omega) = -\hepsiloni(\omega),
\\
&\hat{\chi}(-\omega_1,-\omega_2,-\omega_3) = \hat{\chi}(\omega_1,\omega_2,\omega_3),
\\
&\hat{\chi}(\omega_2,\omega_1,\omega_3)=\hat{\chi}(\omega_1,\omega_2,\omega_3).
\end{split}
\label{sym_chi}
\end{align}

Using the method of multiple scales, we look for a weakly nonlinear asymptotic solution for SPs
of the form
\begin{align}
\begin{split}
\mathbf{E} &= \eps \mathbf{E}_1(x,z, t,\tau) + \eps^3 \mathbf{E}_3(x,z, t, \tau)  + O(\eps^5),
\\
\mathbf{D}  &= \eps \mathbf{D}_1(x,z, t,\tau)  + \eps^3 \mathbf{D}_3(x,z, t,\tau)  + O(\eps^5),
\\
\mathbf{B} &=  \eps^2 \mathbf{B}_2(x,z, t, \tau)  + O(\eps^4),
\\
\mathbf{H} &=  \eps^2 \mathbf{H}_2(x,z, t, \tau)  + O(\eps^4),
\end{split}
\label{asy_exp}
\end{align}
where the `slow' time variable $\tau$ is evaluated at $\eps^2 t$.

The details of the calculation are given in Section~\ref{sec:app}. We summarize the result here.
At the order $\eps$, we find that the leading-order electric field is given by the linearized solution
\begin{align}
\begin{split}
\mathbf{E}_1(x,z, t, \tau) &= \mathbf{\Phi}_1(x,z, \tau) e^{ -i \omega_0 t} + \mathbf{\Phi}_1^*(x,z, \tau) e^{ i \omega_0 t},
\\
\mathbf{\Phi}_1(x,z, \tau )&= \int  \left[\begin{array}{c}
            1        \\
              0  \\
          i \sgn(kz)
     \end{array}  \right]  \hat{A}(k,\tau)   e^{ikx - |kz|} \, dk,
     \end{split}
\label{lin_sol}
\end{align}
where $\omega_0$ satisfies
\[
\hepsilonr^+(\omega_0) + \hepsilonr^-(\omega_0) = 0,
\]
and  $\hat{A}(k,\tau)$ is a complex-valued amplitude function. In particular,
the tangential $x$-component of the electric field on the interface $z=0$ is given by
\eqref{epar} with
\begin{equation}
A(x,\tau) = \int \hat{A}(k,\tau) e^{ikx}\,dk.
\label{FT_A}
\end{equation}

Writing the solution in real form, we have
\[
E_\parallel(x,0,t,\tau) = R(x,\tau) \cos\left(\omega_0 t + \phi(x,\tau)\right),\qquad A = \frac{1}{2} R e^{-i\phi}.
\]
This electric field consists of oscillations of frequency $\omega_0$ with a spatially-dependent amplitude and phase that
vary slowly in time.

The imposition of a solvability condition at the order $\eps^3$ to remove secular terms from the expansion
gives the following spectral equation for $\hat{A}(k,\tau)$:
\begin{align}
\begin{split}
   &\hat{A}_{\tau}(k , \tau)
   =
   i|k|\int T(k,k_2 ,k_3,k_4) \hat{A}^*(k_2,\tau) \hat{A}(k_3,\tau)\hat{A}(k_4,\tau)
\\
&\qquad\qquad\quad \delta(k +k _2 - k_3 - k_4) \, dk_2 dk_3  dk_4
- \gamma  \hat{A}(k,\tau) + \frac{i\nu}{k^2} \hat{A}(k,\tau).
   \end{split}
   \label{spec_eq}
\end{align}
The nonlinear term in \eqref{spec_eq} describes four-wave interactions of wavenumbers $k_2$, $k_3$, $k_4$ into $k= k_3+k_4-k_2$, as
indicated by the Dirac $\delta$-function in the integral.
The interaction coefficient $T$ is given by
\begin{align}
\begin{split}
T(k_1,k_2,k_3,k_4) &= 2a \frac{(k_1k_3 +       |k_1k_3|     )  (k_2k_4 +  |k_2k_4|) } {k_1 k_2 k_3 k_4(|k_1| + |k_2| + |k_3|+ |k_4|)}
\\
&+ b \frac{(k_1k_2 - |k_1k_2|) (k_3k_4 - |k_3k_4|) }{k_1 k_2 k_3 k_4(|k_1| + |k_2| + |k_3|+ |k_4|)},
\end{split}
\label{defT}
\end{align}
where
\begin{align}
\begin{split}
a &= \frac{\hat{\chi}^+(\omega_0, - \omega_0, \omega_0) + \hat{\chi}^-(\omega_0, - \omega_0, \omega_0)}{\hepsilonr'^+(\omega_0) + \hepsilonr'^-(\omega_0)},
\\
b &= \frac{\hat{\chi}^+(\omega_0, \omega_0, - \omega_0) + \hat{\chi}^-(\omega_0, \omega_0, - \omega_0)}{\hepsilonr'^+(\omega_0) + \hepsilonr'^-(\omega_0)}.
\end{split}
\label{def_ab}
\end{align}
Here, the prime denotes a derivative with respect to $\omega$.
The coefficients of damping $\gamma$ and dispersion $\nu$ are given by
\begin{align}
\begin{split}
\gamma &= \frac{\hepsiloni^+(\omega_0) + \hepsiloni^-(\omega_0)}{\hepsilonr'^+(\omega_0) + \hepsilonr'^-(\omega_0)}.
\qquad
\nu =  \frac{1}{2} \mu\omega_0^2\left[\frac{\hepsilonr^+(\omega_0)^2 +  \hepsilonr^-(\omega_0)^2}{   \hepsilonr'^{+}(\omega_0)
+ \hepsilonr'^-(\omega_0)  }\right],
\end{split}
\label{cf_disp}
\end{align}
These coefficients agree with the expansion of the
full linearized dispersion relation
\eqref{gen_disp} as $k\to\infty$, which is
\[
\omega = \omega_0 - \frac{\nu}{k^2} - i\eps^2\gamma + \dots.
\]

The relative strength of the two terms in $T$ is proportional to $b/a$.
The coefficient $\hat{\chi}(\omega_0, - \omega_0, \omega_0)$ appearing in $a$ is associated with cubically nonlinear interactions that produce waves of the same polarization,
while the coefficient $\hat{\chi}(\omega_0, \omega_0, - \omega_0)$ appearing in $b$ is associated with interactions that produce waves of opposite polarization \cite{boyd}. A typical value of the ratio
\[
r = \frac{\hat{\chi}(\omega_0, \omega_0, - \omega_0)}{\hat{\chi}(\omega_0, - \omega_0, \omega_0)}
\]
for nonlinearity due to nonresonant electronic response is $r=1$.
If $r^\pm$ are the values of this ratio in $z>0$ and $z<0$, then
\[
\frac{b}{a} = \frac{r^- + s r^+}{1+s},\qquad
s = \frac{\hat{\chi}^+(\omega_0, - \omega_0, \omega_0)}{\hat{\chi}^-(\omega_0, - \omega_0, \omega_0)}.
\]
For example, if the dielectric in $z > 0$ is a vacuum, then $s = 0$ and $b/a = r^-$.

It is convenient to write the interaction coefficient in \eqref{spec_eq} as an asymmetric function.
We could instead use the symmetrized interaction coefficient
\[
T_s(k_1,k_2, k_3,k_4) = \frac{1}{2}\left[T(k_1,k_2, k_3,k_4) + T(k_1,k_2, k_4,k_3)\right],
\]
which satisfies
\[
T_s(k_1,k_2, k_3,k_4) = T_s(k_2,k_1, k_3,k_4) = T_s(k_1,k_2, k_4,k_3) = T_s(k_3,k_4, k_1,k_2).
\]

The above solution generalizes in a straightforward way to two-dimensional SPs that depend on both tangential space variables
$\mathbf{x} = (x,y)$, with corresponding tangential wavenumber vector $\mathbf{k}=(k,\ell)$.
In that case, we write the equations in terms of a potential variable $\hat{a}$
instead of a field variable $\hat{\mathbf{A}}$, where $\hat{\mathbf{A}}(\mathbf{k},\tau) = i\mathbf{k} \hat{a}(\mathbf{k},\tau)$.
One finds that
\begin{align*}
\mathbf{E}_1(\mathbf{x},z, t, \tau) &= \mathbf{\Phi}_1(\mathbf{x},z, \tau) e^{ -i \omega_0 t} + \mathbf{\Phi}_1^*(\mathbf{x},z, \tau) e^{ i \omega_0 t},
\\
\mathbf{\Phi}_1(\mathbf{x},z, \tau )&= \int  \left[\begin{array}{c}
            i\mathbf{k}  \\
            -|\mathbf{k}|\sgn(z)
     \end{array}  \right]  \hat{a}(\mathbf{k},\tau)   e^{i\mathbf{k}\cdot\mathbf{x} - |\mathbf{k}z|} \, d\mathbf{k},
\end{align*}
where $\hat{a}$ satisfies
\begin{align*}
   &\hat{a}_{\tau}(\mathbf{k} , \tau)
   \\
   &=
   \frac{i}{|\mathbf{k}|}\int
   S(\mathbf{k},\mathbf{k}_2 ,\mathbf{k}_3,\mathbf{k}_4) \hat{a}^*(\mathbf{k}_2,\tau) \hat{a}(\mathbf{k}_3,\tau)\hat{a}(\mathbf{k}_4,\tau)
\\
&\qquad\qquad\quad \delta(\mathbf{k} +\mathbf{k} _2 - \mathbf{k}_3 - \mathbf{k}_4)
   \, d\mathbf{k}_2 d\mathbf{k}_3  d\mathbf{k}_4
  - \gamma  \hat{a}(\mathbf{k},\tau) + \frac{i\nu}{|\mathbf{k}|^2} \hat{a}(\mathbf{k},\tau),
\end{align*}
with
\begin{align*}
S(\mathbf{k}_1,\mathbf{k}_2,\mathbf{k}_3,\mathbf{k}_4) &= 2a \frac{(\mathbf{k}_1\cdot\mathbf{k}_4 +
|\mathbf{k}_1\cdot \mathbf{k}_4|     )  (\mathbf{k}_2\cdot\mathbf{k}_3 +  |\mathbf{k}_2\cdot\mathbf{k}_3|) }
{|\mathbf{k}_1| + |\mathbf{k}_2| + |\mathbf{k}_3|+ |\mathbf{k}_4|}
\\
&+ b \frac{(\mathbf{k}_1\cdot\mathbf{k}_2 - |\mathbf{k}_1\cdot\mathbf{k}_2|) (\mathbf{k}_3\cdot\mathbf{k}_4 - |\mathbf{k}_3\cdot\mathbf{k}_4|) }
{|\mathbf{k}_1| + |\mathbf{k}_2| + |\mathbf{k}_3|+ |\mathbf{k}_4|}.
\end{align*}
In this paper, we restrict our attention to one-dimensional SPs.

\section{Spatial form of the equation}

In this section, we write the spectral equation \eqref{spec_eq}--\eqref{defT} in spatial form for
$A(x,\tau)$ given in \eqref{FT_A}.
To simplify the notation, we do not show the time-dependence of $A$.

First, consider a nonlinear term proportional to
the one in $T$ with coefficient $a$:
\begin{align*}
\hat{F}(k_1) &= \int \frac{  (k_1k_4 +       |k_1k_4|     )  (k_2k_3 +  |k_2k_3|) } {2k_1 k_2 k_3 k_4(|k_1| + |k_2| + |k_3|+ |k_4|)}
\\
&\qquad\qquad\delta_{12-34}\hat{A}^*(k_2) \hat{A}(k_3) \hat{A}(k_4)\, dk_2 dk_3 dk_4,
\end{align*}
where we write
\[
\delta(k_1 + k_2 - k_3 - k_4) = \delta_{12-34}.
\]

The interaction coefficient in this integral is nonzero only if $k_1$, $k_4$ have the same sign and $k_2$, $k_3$ have the same sign. Furthermore, in that case,
we have
\begin{align*}
&\frac{  (k_1k_4 +       |k_1k_4|)  (k_2k_3 +  |k_2k_3|)) } {2k_1 k_2 k_3 k_4(|k_1| + |k_2| + |k_3|+ |k_4|)}
= \frac{2}{|k_1| + |k_2| + |k_3|+ |k_4|}
\\
&\qquad\qquad= \begin{cases} 1/(k_3+k_4) &\mbox{if $k_1,k_4 > 0$ and $k_2,k_3 >0$,}\\
1/(k_2-k_4) &\mbox{if $k_1,k_4 < 0$ and $k_2,k_3 >0$,}\\
-1/(k_2-k_4) &\mbox{if $k_1,k_4 > 0$ and $k_2,k_3 <0$,}\\
-1/(k_3+k_4) &\mbox{if $k_1,k_4 < 0$ and $k_2,k_3 <0$,}
\end{cases}
\end{align*}
 on $k_1 + k_2 = k_3 + k_4$. We decompose $\hat{A}$ into its positive and negative wavenumber components,
\begin{align*}
\hat{A}(k) &= \hat{u}(k) + \hat{v}(k),
\\
\hat{u}(k) &= \hat{\Pr}[A](k) = \begin{cases} \hat{A}(k) &\mbox{if $k > 0$,} \\ 0 &\mbox{if $k < 0$,}\end{cases}
\\
\hat{v}(k) &= \hat{\Qr}[A](k) = \begin{cases} 0 &\mbox{if $k > 0$,} \\ \hat{A}(k) &\mbox{if $k < 0$.}\end{cases}
\end{align*}
The corresponding spatial decomposition is $A = u + v$, where $u=\Pr[A]$, $v=\Qr[A]$ are given by \eqref{def_uv}.

It then follows that
\begin{align*}
\hat{F}(k_1) &= \hat{\Pr}\int\delta_{12-34} \frac{\hat{u}^*(k_2) \hat{u}(k_3)\hat{u}(k_4)}{k_3+k_4}\, dk_2 dk_3 dk_4
\\
&+ \hat{\Qr}\int\delta_{12-34} \frac{\hat{u}^*(k_2) \hat{u}(k_3)\hat{v}(k_4)}{k_2-k_4}\, dk_2 dk_3 dk_4
\\
&- \hat{\Pr}\int\delta_{12-34} \frac{\hat{v}^*(k_2) \hat{v}(k_3)\hat{u}(k_4)}{k_2-k_4}\, dk_2 dk_3 dk_4
\\
&-\hat{\Qr}\int\delta_{12-34} \frac{\hat{v}^*(k_2) \hat{v}(k_3)\hat{v}(k_4)}{k_3+k_4}\, dk_2 dk_3 dk_4.
\end{align*}
The integral in the first term,
\[
\hat{I}(k_1) = \int\delta_{12-34} \frac{\hat{u}^*(k_2) \hat{u}(k_3)\hat{u}(k_4)}{k_3+k_4}\, dk_2 dk_3 dk_4,
\]
has inverse Fourier transform
\begin{align*}
I(x) &= \int \hat{I}(k_1) e^{ik_1x}\, dk_1
\\
&= \int\hat{u}^*(k_2) e^{-ik_2 x} \frac{\hat{u}(k_3) \hat{u}(k_4) e^{i(k_3+k_4)x}}{k_3+k_4}\, dk_2 dk_3 dk_4
\\
&= i u^* \partial_x^{-1}(u^2).
\end{align*}
The other terms in $\hat{F}$ are treated similarly, and we find that
\[
F(x) = i\Pr\left[u^* \partial_x^{-1}(u^2) + v \partial_x^{-1}(uv^*)\right] - i\Qr\left[u \partial_x^{-1}(u^* v) +v^* \partial_x^{-1}(v^2) \right].
\]

The nonlinear term proportional to $b$,
\begin{align*}
\hat{G}(k_1) &= \int \frac{ (k_1k_2 - |k_1k_2|) (k_3k_4 - |k_3k_4|) }{2k_1 k_2 k_3 k_4(|k_1| + |k_2| + |k_3|+ |k_4|)}
\\
&\qquad\qquad\delta_{12-34}\hat{A}^*(k_2) \hat{A}(k_3) \hat{A}(k_4)\, dk_2 dk_3 dk_4,
\end{align*}
is zero unless  $k_1$, $k_2$ and $k_3$, $k_4$ have opposite signs, in which case
\begin{align*}
&\frac{  (k_1k_2 - |k_1k_2|) (k_3k_4 - |k_3k_4|) } {2k_1 k_2 k_3 k_4(|k_1| + |k_2| + |k_3|+ |k_4|)}
\\
&\qquad\qquad
= \begin{cases}
1/(k_2-k_3) &\mbox{if $k_1 < 0$, $k_2 > 0$, $k_3 < 0$, $k_4 >0$,}\\
1/(k_2-k_4) &\mbox{if $k_1 < 0$, $k_2 >0$, $k_3 > 0$, $k_4 < 0$,}\\
-1/(k_2-k_4) &\mbox{if $k_1 > 0$, $k_2 <0$, $k_3 < 0$, $k_4 > 0$,}\\
-1/(k_2-k_3) &\mbox{if $k_1 > 0$, $k_2 <0$, $k_3 > 0$, $k_4 < 0$,}
\end{cases}
\end{align*}
on $k_1 + k_2 = k_3 + k_4$. It then follows that
\[
G(x) = 2i\Pr\left[v \partial_x^{-1}(u  v^*)\right] - 2i\Qr\left[u \partial_x^{-1}(u^* v)\right].
\]

Projecting \eqref{spec_eq} onto positive and negative wavenumbers and using the previous equations to take the inverse Fourier transform, we find that
\begin{align}
\begin{split}
&u_\tau = i\partial_x\Pr\left[\alpha u^* \partial_x^{-1}(u^2) + \beta v \partial_x^{-1}(uv^*)\right] - \gamma u - i\nu \partial_x^{-2} u,
\qquad \Pr[u]=u,
\\
&v_\tau = i\partial_x\Qr\left[\alpha v^* \partial_x^{-1}(v^2)
+\beta u \partial_x^{-1}(u^* v)\right]-\gamma v  - i\nu \partial_x^{-2} v,
\qquad \Qr[v] = v,
\end{split}
\label{uv_eq_disp}
\end{align}
where $\gamma$, $\nu$ are given by \eqref{cf_disp} and $\alpha$, $\beta$ are given by
\begin{equation}
\alpha = 4 a, \qquad \beta = 4(a+b),
\label{def_alpa-beta}
\end{equation}
where $a$, $b$ are defined in \eqref{def_ab}.

In the presence of damping, SPs require external forcing to maintain their energy. Direct forcing by electromagnetic radiation is not feasible, since
the wavenumber, or momentum, of an electromagnetic wave with same frequency as the SP is smaller than that of the SP. Instead, SPs are typically forced
by an evanescent wave (in the Otto or Kretschmann configuration). At least heuristically, this forcing can be modeled by the inclusion of nonhomogeneous
terms in \eqref{uv_eq_disp}. For example, a steady pattern of forcing at a slightly detuned frequency $\omega = \omega_0 + \eps^2 \Omega$ leads to
equations of the form
\begin{align*}
&u_\tau = i\partial_x\Pr\left[\alpha u^* \partial_x^{-1}(u^2) + \beta v \partial_x^{-1}(uv^*)\right] - \gamma u - i\nu \partial_x^{-2} u
+ f(x) e^{-i\Omega\tau},
\\
&v_\tau = i\partial_x\Qr\left[\alpha v^* \partial_x^{-1}(v^2)
+\beta u \partial_x^{-1}(u^* v)\right]-\gamma v  - i\nu \partial_x^{-2} v + g(x) e^{-i\Omega\tau},
\end{align*}
where $\Pr[f]=f$, $\Qr[g]=g$. We do not study the behavior of the resulting damped, forced system in this paper.

\section{Short-time existence of smooth solutions}
\label{sec:short_time}

For definiteness, we consider spatially periodic solutions of the asymptotic equations
\eqref{uv_eq_disp} with zero mean. For $s\in \Rl$, let $\dot{H}^s(\Ts)$ denote the Sobolev space of $2\pi$-periodic, zero-mean
functions (or distributions)
\[
f(x) = \sum_{n\in \Ir^*} \hat{f}(n) e^{inx},
\]
where $\Ir^* = \Ir\setminus\{0\}$, such that
\[
\|f\|_s = \| |\partial_x|^s f\|_{L^2} = \left(\sum_{n\in\Ir^*} |n|^{2s} |\hat{f}(n)|^2\right)^{1/2} < \infty.
\]
We also use the Wiener-algebra norm
\begin{equation}
\|f\|_A = \sum_{n\in\Ir^*} |\hat{f}(n)|,
\label{Anorm}
\end{equation}
which satisfies $\|f\|_A\lesssim \|f\|_s$ for $s>1/2$ and
\begin{equation}
\|fg\|_A \le \|f\|_A \|g\|_A,\qquad
\|fg\|_s \le \|f\|_A \|g\|_s + \|f\|_s \|g\|_A.
\label{prod_est}
\end{equation}

The projections of $f$ onto positive and negative wavenumbers are given by
\[
\Pr[f](x) =  \sum_{n=1}^\infty \hat{f}(n) e^{inx},\qquad \Qr[f](x) =  \sum_{n=-\infty}^{-1} \hat{f}(n) e^{inx}.
\]
These projections satisfy
\[
\int_{\Ts} \Pr[f]g \, dx = \int_{\Ts}f \Qr[g]\, dx
\]
for all zero-mean, $L^2$-functions $f$, $g$. We denote the corresponding projected Sobolev spaces by
\begin{equation}
{H}_+^s(\Ts) = \left\{ u\in \dot{H}^s(\Ts) : \Pr[u]=u\right\},
\quad
{H}_-^s(\Ts) = \left\{ v\in \dot{H}^s(\Ts) : \Qr[v]=v\right\}.
\label{def_hs}
\end{equation}

\begin{theorem}
\label{th:short_time}
Suppose that $s>1/2$ and $f\in {H}_+^s(\Ts)$, $g \in {H}_-^s(\Ts)$. Then there exists $T=T(\|f\|_s,\|g\|_s) > 0$ such that
the initial-value problem
\begin{align*}
&u_\tau = i\partial_x\Pr\left[\alpha u^* \partial_x^{-1}(u^2) + \beta v \partial_x^{-1}(uv^*)\right] - \gamma u - i\nu \partial_x^{-2} u,
\qquad \Pr[u]=u,
\\
&v_\tau = i\partial_x\Qr\left[\alpha v^* \partial_x^{-1}(v^2)
+\beta u \partial_x^{-1}(u^* v)\right]-\gamma v  - i\nu \partial_x^{-2} v,
\qquad \Qr[v] = v,
\\
&u(0) = f,\qquad v(0) = g
\end{align*}
has a unique solution with
\[
u \in C([-T,T]; {H}_+^s),\qquad v \in C([-T,T]; {H}_-^s).
\]
Moreover, this $H^s$-solution breaks down as $\tau\uparrow T_* > 0$  only if
\[
\int_0^\tau \left\{\|u\|_A^2(s) + \|v\|_A^2(s) \right\}\, ds \uparrow \infty \qquad \mbox{as $\tau\uparrow T_*$}.
\]
\end{theorem}

\begin{proof}
Consider the nonlinear term
\[
F(u) = \partial_x\Pr[u^*\partial_x^{-1}(u^2)].
\]
By expanding the derivative and using the fact that $\Pr[u^*] = 0$, we can write this term as
\[
F(u) = \Pr[|u|^2 u] + [\Pr, \partial_x^{-1}(u^2)] u_x^*,
\]
where $[\Pr,w] = \Pr w - w\Pr$ denotes a commutator.
The use of \eqref{prod_est} and Lemma~\ref{lem:com_est}, proved in Section~\ref{sec:lem}, then implies that
\begin{align*}
\|F(u_1) - F(u_2)\|_s &\lesssim \left(\|u_1\|_A^2 + \|u_2\|_A^2\right) \|u_1-u_2\|_s
\\
&\lesssim \left(\|u_1\|_s^2 + \|u_2\|_s^2\right) \|u_1-u_2\|_s
\end{align*}
for $s > 1/2$, so $F$ is Lipschitz continuous on ${H}^s_+$. A similar computation applies to the other nonlinear terms since, for example,
\[
\Pr\left[v_x \partial_x^{-1}(uv^*)\right] = [\Pr, \partial_x^{-1}(uv^*)] v_x.
\]
It follows that the right-hand side of
\eqref{uv_eq_disp} is a Lipschitz-continuous function of $(u,v)$ on ${H}_+^s \times {H}_-^s$ for $s > 1/2$,
so the Picard theorem implies local existence.

For simplicity, we prove the breakdown criterion for the unidirectional equation \eqref{u_eq}.
A similar proof applies to the general system \eqref{uv_eq_disp}.
If $s > 1/2$ and $u(x,\tau)$ is an ${H}^s_+$-solution of \eqref{u_eq}, then
\begin{align*}
\frac{d}{d\tau} \int |\partial_x|^su^* \cdot |\partial_x|^su\, dx &=
2 \Im\int  |\partial_x|^su^* \cdot|\partial_x|^s\left\{|u|^2 u + [\Pr, \partial_x^{-1}(u^2)] u_x^*\right\}\, dx.
\end{align*}
By the Cauchy-Schwartz inequality and \eqref{prod_est}, we have
\[
\left|\int  |\partial_x|^su^* \cdot|\partial_x|^s(|u|^2 u)\,dx\right|
\le \|u\|_s \left\||u|^2 u\right\|_s \le \|u\|_A^2 \|u\|_s^2.
\]
Similarly, by Lemma~\ref{lem:com_est}, we have
\begin{align*}
\left|\int  |\partial_x|^su^* \cdot|\partial_x|^s [\Pr, \partial_x^{-1}(u^2)] u_x^*\,dx\right|
& \le\|u\|_s \left\|[\Pr, \partial_x^{-1}(u^2)] u_x^*\right\|_s\lesssim \|u\|_A^2 \|u\|_s^2.
\end{align*}
It follows that
\[
\frac{d}{d\tau} \|u\|_s^2 \lesssim \|u\|_A^2 \|u\|_s^2,
\]
so Gronwall's inequality implies that $\|u\|_s$ remains bounded so long as
\[
\int_0^\tau \|u\|_A^2(s)\, ds
\]
remains finite.
\end{proof}

\section{Hamiltonian Form}
\label{sec:ham}

The Kramers-Kronig causality relations imply that the dispersion of electromagnetic waves is accompanied by some kind of dissipation. Consequently, the macroscopic Maxwell equations do not have a Hamiltonian structure. Nevertheless, as observed by Zakharov \cite{zakharov}, in regimes where dissipation can be neglected the amplitude equations of nonlinear optics are invariably Hamiltonian, and that is the case here.

When $\gamma=0$, equation \eqref{uv_eq_disp} has the Hamiltonian form
\begin{align}
\begin{split}
&u_\tau =  \mathbf{J}  \Pr\left[\frac{\delta\mathcal{H}}{\delta u^*}\right],
\qquad
v_\tau =  \mathbf{J} \Qr\left[\frac{\delta\mathcal{H}}{\delta v^*}\right],
\end{split}
\label{uv_ham_eq}
\end{align}
subject to the constraints $\Pr[u]=u$, $\Qr[v]=v$ (see \cite{morrison} for further discussion of constraints), where
\[
\mathbf{J} = i|\partial_x|
\]
is a skew-adjoint Hamiltonian operator, and
the Hamiltonian is given by
\begin{align*}
\mathcal{H}(u,v, u^*,v^*) &= \int \left\{\frac{1}{2}\alpha (u^*)^2|\partial_x|^{-1}(u^2)+ \beta u^* v |\partial_x|^{-1}(uv^*)\right.
\\
&\qquad\qquad\left.+ \frac{1}{2}\alpha (v^*)^2 |\partial_x|^{-1}(v^2) + \nu u^*|\partial_x|^{-3} u + \nu v^*|\partial_x|^{-3}v  \right\} \, dx.
\end{align*}
Here, the operator $|\partial_x|$ is defined spectrally by
\[
|\partial_x|(e^{ikx}) = |k| e^{ikx},
\]
so that
\[
|\partial_x|^{-1} u = i\partial_x^{-1} u,\qquad |\partial_x|^{-1} v = -i\partial_x^{-1} v.
\]

Equivalently, the Hamiltonian form of the asymptotic equation for $A=u+v$ is
\[
A_\tau = \mathbf{J} \left[\frac{\delta\mathcal{H}}{\delta A^*}\right].
\]
The Hamiltonian form of the spectral equation \eqref{spec_eq}, with $\gamma=0$, for $\hat{A}$ is
\[
\hat{A}_\tau = \hat{\mathbf{J}} \left[\frac{\delta\mathcal{H}}{\delta \hat{A}^*}\right],
\]
where $\hat{\mathbf{J}} = i|k|$, and
\begin{align}
\begin{split}
\mathcal{H}(\hat{A}, \hat{A}^*) &= \frac{1}{2}\int T(k_1, k_2, k_3, k_4)
\hat{A}^*(k_1) \hat{A}^*(k_2) \hat{A}(k_3) \hat{A}(k_4)
\\
&\qquad\qquad \delta(k_1 + k_2 - k_3 -k_4)\, dk_1 dk_2 dk_3 dk_4
\\
&+ \nu\int \frac{\hat{A}^*(k)\hat{A}(k)}{|k|^3}\, dk.
\end{split}
\label{spec_ham}
\end{align}

In addition to the energy $\mathcal{H}$, other conserved quantitites for \eqref{uv_ham_eq} are the right and left actions (associated with the invariance of $\mathcal{H}$
under phase changes $u\mapsto e^{i\phi} u$ and $v \mapsto e^{i\psi} v$)
\begin{equation}
\mathcal{S} = \int u^* |\partial_x|^{-1} u\, dx, \qquad \mathcal{T} = \int v^* |\partial_x|^{-1} v\, dx, 
\label{uv_action}
\end{equation}
and the momentum (associated with the invariance of $\mathcal{H}$
under translations $x\mapsto x+h$)
\begin{equation}
\mathcal{P} = \int \left\{|u|^2 - |v|^2\right\}\, dx.
\label{uv_mom}
\end{equation}
These conservation laws lead to global a priori estimates for the $H^{-1/2}$-norms of $u$, $v$, $u^2$, $v^2$, and $u^*v$ (assuming that
$\alpha$, $\beta$ have the same sign), but these estimates do not appear to be sufficient to imply the existence of global weak solutions.

From \eqref{uv_mom}, the unidirectional equation with $v=0$,
\begin{equation}
u_\tau = i \partial_x\Pr[u^* \partial_x^{-1} (u^2)]- i\nu \partial_x^{-2} u,
\label{uni_disp}
\end{equation}
has, in addition, an a priori estimate for the $L^2$-norm of $u$;
this estimate does not apply to the system because the momenta of the left and right moving waves have opposite signs.
In the next section, we show that this $L^2$-estimate is sufficient to imply the existence of global weak solutions of \eqref{uni_disp}.

Expanding the derivative in \eqref{uni_disp} and setting $\nu=0$, we get
\[
u_\tau =  i \Pr[|u|^2 u)] + i \Pr[ u^*_x \partial_x^{-1} (u^2)].
\]
If we neglect the second term on the right-hand side of this equation, then
we obtain the Szeg\"o equation \eqref{szego_eq}, up to a difference in sign.
The sign-difference is inessential, since the $CP$-transformation $u\mapsto u^*$ and $x\mapsto -x$
maps $u_\tau = i\Pr[|u|^2 u]$ into $iu_\tau = \Pr[|u|^2 u]$; analogous transformations apply to \eqref{uv_eq_disp} and \eqref{uni_disp}.

The Szeg\"o equation is Hamiltonian, but it has a different, complex-canonical,  constant Hamiltonian structure
from \eqref{uni_disp}:
\[
i u_\tau = \Pr\left[\frac{\delta \mathcal{H}}{\delta u^*}\right],\qquad\mathcal{H}(u,u^*) = \frac{1}{2}\int (u^*)^2 u^2\, dx.
\]
The momentum for this equation is the $H^{1/2}$-norm of $u$,
\[
\mathcal{P} = \int u^*|\partial_x| u\, dx,
\]
which is scale-invariant and critical for \eqref{szego_eq}. As a result, the Szeg\"o equation has unique global smooth $H^s$-solutions for $s \ge 1/2$ and global weak $L^2$-solutions \cite{szego}.

\section{Global existence of weak solutions for the unidirectional equation}

We consider spatially periodic solutions of the unidirectional equation \eqref{u_eq}. The same results apply to the damped, dispersive version of the equation.

For $1\le p \le \infty$, we denote the $L^p$-Hardy space of zero-mean functions on $\Ts$ with vanishing negative Fourier coefficients by
\begin{align*}
&L^p_+(\Ts) = \left\{f\in L^p(\Ts) : \mbox{$\hat{f}(n) = 0$ for $n\le 0$}\right\},
\\
&\hat{f}(n) = \frac{1}{2\pi}\int_{\Ts} f(x) e^{-inx}\, dx.
\end{align*}
We denote the corresponding first-order $L^p$-Sobolev space by
\[
W_+^{1,p}(\Ts) = \left\{f\in L_+^p(\Ts) : \mbox{$f_x\in L_+^p(\Ts)$}\right\},
\]
where, by the Poincar\'e inequality for zero-mean functions, we use the norm
\[
\|f\|_{{W}_+^{1,p}} = \|f_x\|_{L^p}.
\]
We continue to use the notation in \eqref{def_hs} for the $L^2$-Sobolev spaces of order $s\in \Rl$.

\begin{definition}
\label{def:weak_sol}
A function $u: \Rl \to L^2_+(\Ts)$ is a weak solution of \eqref{u_eq} if
\begin{equation}
\frac{d}{d\tau}\int \phi^* u\, dx  = - i\int \phi^*_x u^*\partial_x^{-1}(u^2) \, dx
\label{weak_form}
\end{equation}
for every $\phi\in W^{1,p}_+(\Ts)$ with $p > 2$,
in the sense of distributions on $\mathcal{D}'(\Rl)$.
\end{definition}

We remark that, by Sobolev embedding,
\begin{equation}
\|u^*\partial_x^{-1}(u^2) \|_{L^2} \le \|\partial_x^{-1}(u^2)\|_{L^\infty} \|u\|_{L^2}
\lesssim \|u^2\|_{L^1} \|u\|_{L^2} \le \|u\|_{L^2}^3,
\label{non_est}
\end{equation}
so the nonlinear term makes sense.

Next, we prove the global existence of weak solutions; the weak continuity of the nonlinear term depends crucially
on the fact that $u$ contains only positive Fourier components.

\begin{theorem}
\label{th:global_weak}
If $f\in L_+^2(\Ts)$, then there exists a global weak solution
of the initial value problem
\begin{align*}
&u_\tau = i \partial_x\Pr\left[u^* \partial_x^{-1}(u^2)\right],\quad \Pr[u]=u,
\\
&u(0)=f
\end{align*}
with $u\in L^\infty(\Rl;L_+^2)$ and  $u_\tau \in L^\infty(\Rl; H_+^{-1})$.
\end{theorem}

\begin{proof}
We use a Galerkin method. Let $\Pr_N$ denote the projection onto the first $N$ positive Fourier modes,
\[
\Pr_N\left[\sum_{n=-\infty}^\infty \hat{f}(n) e^{inx}\right] = \sum_{n=1}^N \hat{f}(n) e^{inx},
\]
and let $u_N$ be the solution of the ODE
\begin{align}
\begin{split}
&u_{N\tau} = i\partial_x\Pr_N\left[ u^*_N\partial_x^{-1}(u_N^2)\right],\qquad \Pr_N[u_N] = u_N,
\\
&u_N(0) = \Pr_N[f].
\end{split}
\label{Neq}
\end{align}
Then
\begin{align*}
\frac{d}{d\tau} \int u_N^* u_N \, dx &=
2\Im \int u_N^* \partial_x\Pr_N\left[u^*_N\partial_x^{-1}(u_N^2)\right]\, dx
\\
&= - \Im \int\partial_x(u_N^{*2}) \partial_x^{-1}(u_N^{2})\, dx
\\
&=0,
\end{align*}
so the solution $u_N : \Rl\to L^2_+(\Ts)$ exists globally in time, and
\begin{equation}
\|u_N\|_{L^2}  = \|\Pr_N f\|_{L^2} \le \|f\|_{L^2}.
\label{uN_L2}
\end{equation}
Moreover, as in \eqref{non_est}, we have
\begin{equation}
\left\|u_N^* \partial_x^{-1} (u_N^2)\right\|_{L^2} \lesssim \|f\|_{L^2}^3.
\label{nonN_est}
\end{equation}

Fix an arbitrary $T > 0$. Then it follows from \eqref{Neq}--\eqref{nonN_est} that
\begin{align}
\begin{split}
\left\{u_N : N\in \Nl\right\}\qquad &\mbox{is bounded in
$L^\infty(-T,T; L_+^2)$},
\\
\left\{u_{N\tau} : N \in \Nl\right\}
\qquad
&\mbox{is bounded in $L^\infty(-T,T;H_+^{-1})$}.
\end{split}
\label{uN_bounds}
\end{align}
Thus, by the Banach-Alaoglu theorem, we can extract a weak*-convergent subsequence, still denoted by $\{u_N\}$, such that
\begin{align}
\begin{split}
u_N \overset{*}{\rightharpoonup} u \qquad &\mbox{in $L^\infty(-T,T; L_+^2)$ as $N\to \infty$},
\\
u_{N\tau} \overset{*}{\rightharpoonup} u_\tau \qquad &\mbox{in $L^\infty(-T,T; H^{-1}_+)$ as $N\to \infty$}.
\end{split}
\label{uN_con}
\end{align}

In order to take the limit of the nonlinear term in \eqref{Neq}, we consider
\[
w_N = \partial_x^{-1}(u_N^2),
\]
which satisfies the estimates
\begin{align}
&\|w_N\|_{{W}_+^{1,1}} = \|u_N^2\|_{L^1} \le \|f\|_{L^2}^2,
\label{wN_W11}
\\
&\|w_N\|_{L^\infty} \lesssim \|w_N\|_{W_+^{1,1}} \le \|f\|_{L^2}^2.
\label{wN_Linf}
\end{align}

The time-derivative of $w_N$ satisfies
\[
w_{N\tau} = 2\partial_x^{-1}\left(u_N u_{N\tau}\right) = 2i \partial_x^{-1}\left(u_N\partial_x\Pr_N\left[ u^*_N w_N\right]\right).
\]
Using Lemma~\ref{lem:w}, proved in Section~\ref{sec:lem}, together with \eqref{uN_L2} and \eqref{wN_Linf}, we have for $s > 3/2$ that
\begin{align*}
\left\|u_N\partial_x\Pr_N\left[ u^*_N w_N\right]\right\|_{H^{-s}} &\lesssim \|u_N\|_{L^2} \| u^*_N w_N\|_{L^2}
\\
&\lesssim \|w_N\|_{L^\infty}\|u_N\|^2_{L^2}
\\
&\lesssim \|f\|_{L^2}^4.
\end{align*}
It follows from this estimate and the one in \eqref{wN_W11} that, for $s > 1/2$,
\begin{align}
\begin{split}
\left\{ w_N : N\in \Nl\right\}\qquad
&\mbox{is bounded in $L^\infty(-T,T; W_+^{1,1})$},
\\
\left\{w_{N\tau} : N\in \Nl\right\}\qquad
&\mbox{is bounded in $L^\infty(-T,T; H_+^{-s})$}.
\label{wN_bound}
\end{split}
\end{align}

The space $W_+^{1,1}(\Ts)$ is compactly embedded in $L^q_+(\Ts)$ for any $1\le q < \infty$,
and $L^q_+(\Ts)$ is continuously embedded in $H_+^{-s}(\Ts)$ for $s > {1/2}$,
so the Aubin-Lions-Simon theorem (see \cite{boyer}, for example)
and \eqref{wN_bound} imply that
\begin{equation}
\left\{w_{N} : N\in \Nl\right\}\qquad
\mbox{is strongly precompact in $C(-T,T; L^q)$}.
\label{wN_compact}
\end{equation}
We can therefore extract a further subsequence from $\{u_N\}$ such that the corresponding subsequence $\{w_N\}$ converges strongly, meaning that
\begin{equation}
\partial_x^{-1}(u_N^2) \rightarrow w \qquad \mbox{strongly in $C(-T,T; L^q)$ as $N\to \infty$}.
\label{wN_strong}
\end{equation}
By weak-strong convergence, we get from \eqref{uN_con} and \eqref{wN_strong} that
\[
u_N^*\partial_x^{-1}(u_N^2) \overset{*}{\rightharpoonup} u^* w\qquad \mbox{in $L^\infty(-T, T; L^r)$ as $N\to \infty$},
\]
where $2 \le q <\infty$ and $r = 2q/(q+2) < 2$.
Taking the limit of the weak form of \eqref{Neq} as $N\to \infty$ for test functions $\phi \in W^{1,p}$ with $p=r' > 2$,
we find that $(u,w)$ satisfies the weak form of the equation
\[
u_\tau = i \partial_x \Pr[u^*w],\qquad \Pr[u]=u.
\]

To show that $u$ is a weak solution of \eqref{u_eq}, it remains to verify that $w = \partial_x^{-1}(u^2)$.
The bounds in \eqref{uN_bounds} and the Aubin-Lions-Simon theorem imply that
\[
u_N \rightarrow u \qquad \mbox{strongly in $C(-T,T;H_+^{-s})$ as $N\to \infty$}
\]
for $0<s\le 1$.
It follows that the Fourier coefficients  of ${u}_N$ converge to those of $u$
for every $n\in \Nl$ and $t\in [-T,T]$, since $e^{-inx} \in H_{-}^s = (H_+^{-s})'$ and
\begin{align*}
&\hat{u}_N(n,t) = \frac{1}{2\pi} \int_{\Ts} u_N(x,t) e^{-inx}\, dx
\to \frac{1}{2\pi} \int_{\Ts} u(x,t) e^{-inx}\, dx
= \hat{u}(n,t),
\end{align*}
for every $t \in [-T,T]$. Since the negative Fourier coefficients of $u_N$ are zero,
the Fourier coefficients  of ${u}_N^2$ are given by finite sums of products of the Fourier coefficients of $\hat{u}_N$, so they
converge to the Fourier coefficients of $u^2$:
\begin{align*}
&\widehat{(u_N^2)}(n,t) = \sum_{k=1}^{n-1} \hat{u}_N(n-k,t) \hat{u}_N(k,t)
\to \sum_{k=1}^{n-1} \hat{u}(n-k,t) \hat{u}(k,t)
= \widehat{(u^2)}(n,t).
\end{align*}

Similarly, the strong convergence in \eqref{wN_strong} implies that
\[
\frac{1}{in} \widehat{(u_N)^2}(n,t) \to \hat{w}(n,t)
\]
 for every $n\in \Nl$ and $t\in [-T,T]$. Thus,
\[
\frac{1}{in} \widehat{(u^2)}(n,t) = \hat{w}(n,t),
\]
which implies that $w = \partial_x^{-1}(u^2)$.

In addition, since $u\in C(-T,T; H_+^{-s})$
for $0<s\le 1$, the solution takes on the initial condition
$u(0) = f$ in this $H^{-s}$-sense.

Finally, we get a global weak solution $u : \Rl \to L^2_+(\Ts)$ by a standard diagonal argument.
First, we construct a weak solution on the time-interval $(-1,1)$ as the limit as $N\to\infty$ of a subsequence $\{u^1_N\}$ of approximate solutions.
Next, for each $T\in \Nl$, we construct a weak solution on $(-T-1, T+1)$
by extracting a convergent subsequence $\{u^{T+1}_N\}$ of approximate solutions from the sequence $\{u^T_N\}$ of approximate solutions
that converges to the weak solution
on $(-T,T)$. Then the diagonal sequence of approximate solutions $\{u_N^N\}$ converges
on every time-interval $(-T,T)$ to a global weak solution $u$.
\end{proof}

Our proof gives weak solutions that satisfy \eqref{weak_form}
for test functions $\phi \in W^{1,p}(\Ts)$ with $p>2$, but it does not show that this condition holds when $p=2$. The proof fails for $p=2$
because $W^{1,1}(\Ts)$ is not compactly embedded in $C(\Ts)$, whereas
it is compactly embedded in $L^q(\Ts)$ for $q<\infty$ \cite{brezis}.

We do not consider the uniqueness of these weak solutions, which seems unlikely without the imposition of further admissibility conditions.
Numerical simulations, shown in the next section, suggest that \eqref{u_eq} also has global smooth solutions, but it is unclear how to prove this
in the absence of stronger a priori estimates.

\section{Numerical solutions}

In this section, we show some numerical solutions of the asymptotic equations.
We computed spatially $2\pi$-periodic solutions by a pseudo-spectral method and the method of lines,
typically using a forth-order Runge-Kutta method in time.

In our computations, we used a total of $M=2^m$ positive and negative Fourier modes, for various values of $m$, with the highest half of the modes truncated
to avoid aliasing errors. The resulting effective resolution for the numerical approximation
\[
u(x,\tau) \approx \sum_{k=1}^N \hat{u}(k,\tau) e^{ikx},\qquad v(x,\tau) \approx \sum_{k=1}^{N} \hat{v}(-k,\tau) e^{-ikx}
\]
is $N=2^n$ with $n=m-2$. The surface plots shown here are computed with $N = 2^{11}$ modes.

The numerical solutions indicate that nonlinear effects lead to strong spatial focusing of SPs. The focusing appears to be
most extreme for real initial data, when the oscillations of the SP are in phase at different spatial locations, and we show solutions for this case.
In deriving the asymptotic equations, we neglect any spatial dispersion of the media, which could become important when the
SP focuses, but we do not consider its effects here.

In Figure~\ref{fig:uv}, we show a solution of the system \eqref{uv_eq} with the initial data
\begin{equation}
u(x,0) = e^{ix} + 2 e^{2i(x+2\pi^2)},\qquad v(x,0) = u^*(x,0).
\label{uv_ic}
\end{equation}
The solution focuses strongly. As shown in Figure~\ref{fig:uv_max}, the maximum value of the field-strength variable,
\[
\|A\|_\infty(\tau) = \max_{x\in \Ts} |A(x,\tau)|,
\]
increases by a factor of over $20$ from
$\|A\|_\infty\approx 5.6$ at $\tau=0$ to $\|A\|_\infty\gtrsim 120$ at $\tau = 0.8$, and perhaps blows up in finite time.
Figure~\ref{fig:uv_A} shows the $A$-norm of the solution, which by Theorem~\ref{th:short_time} controls its smoothness. Even with the use of $N = 2^{24} \approx 1.7 \times 10^7$ Fourier modes, it is unclear whether or not the $A$-norm blows up, and further numerical studies are required.

\begin{figure}
\includegraphics[width=4in]{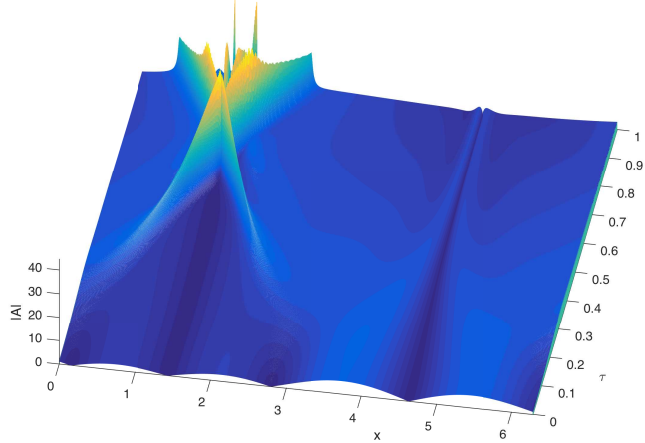}
\caption{A numerical solution of \eqref{uv_eq} for $|A|=|u+v|$ with $\alpha=1$, $\beta=2$ and initial data \eqref{uv_ic}.}
\label{fig:uv}
\end{figure}

\begin{figure}
\includegraphics[width=4in]{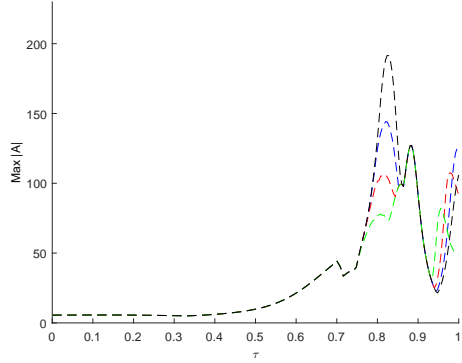}
\caption{Plot of the  maximum value of $|A|$ for the initial value problem shown in Figure~\ref{fig:uv} as a function of time $\tau$ with $N = 2^n$ modes where, $n= 21, 22, 23, 24$.}
\label{fig:uv_max}
\end{figure}

\begin{figure}
\includegraphics[width=4in]{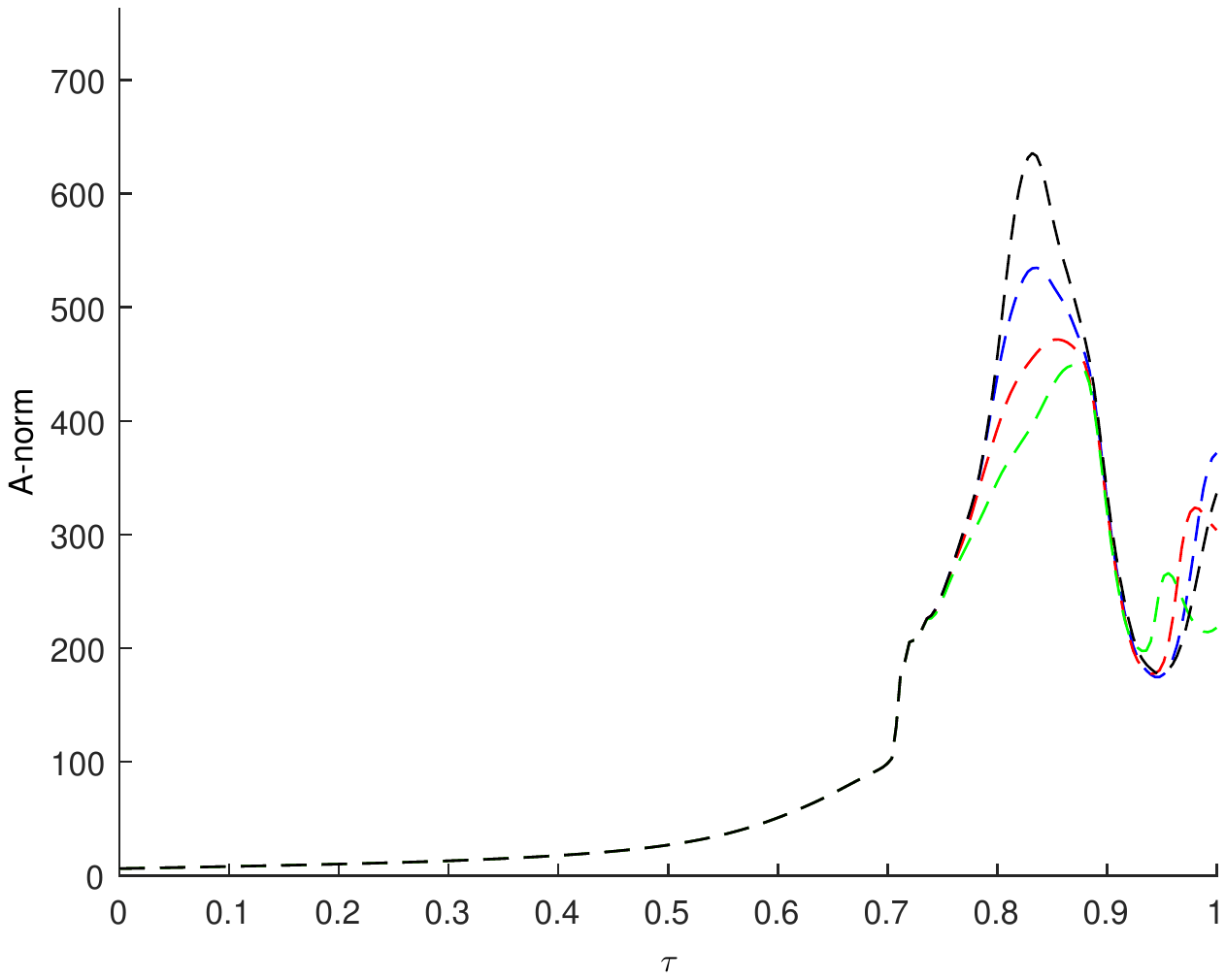}
\caption{Plot of the $A$-norm for the initial value shown in Figure~\ref{fig:uv} as a function of time $\tau$ with $N = 2^n$ modes, where $n= 21, 22, 23, 24$.}
\label{fig:uv_A}
\end{figure}

In Figure~\ref{fig:u}, we show a solution of the unidirectional equation \eqref{u_eq} with the initial data
\begin{equation}
u(x,0) = e^{ix} + 2 e^{2i(x+2\pi^2)}.
\label{u_ic}
\end{equation}
The solution also focuses strongly. However, despite this focusing, the $A$-norm of the solution
appears to remain finite, as shown in Figure~\ref{fig:u_a1}.

\begin{figure}
\includegraphics[width=4in]{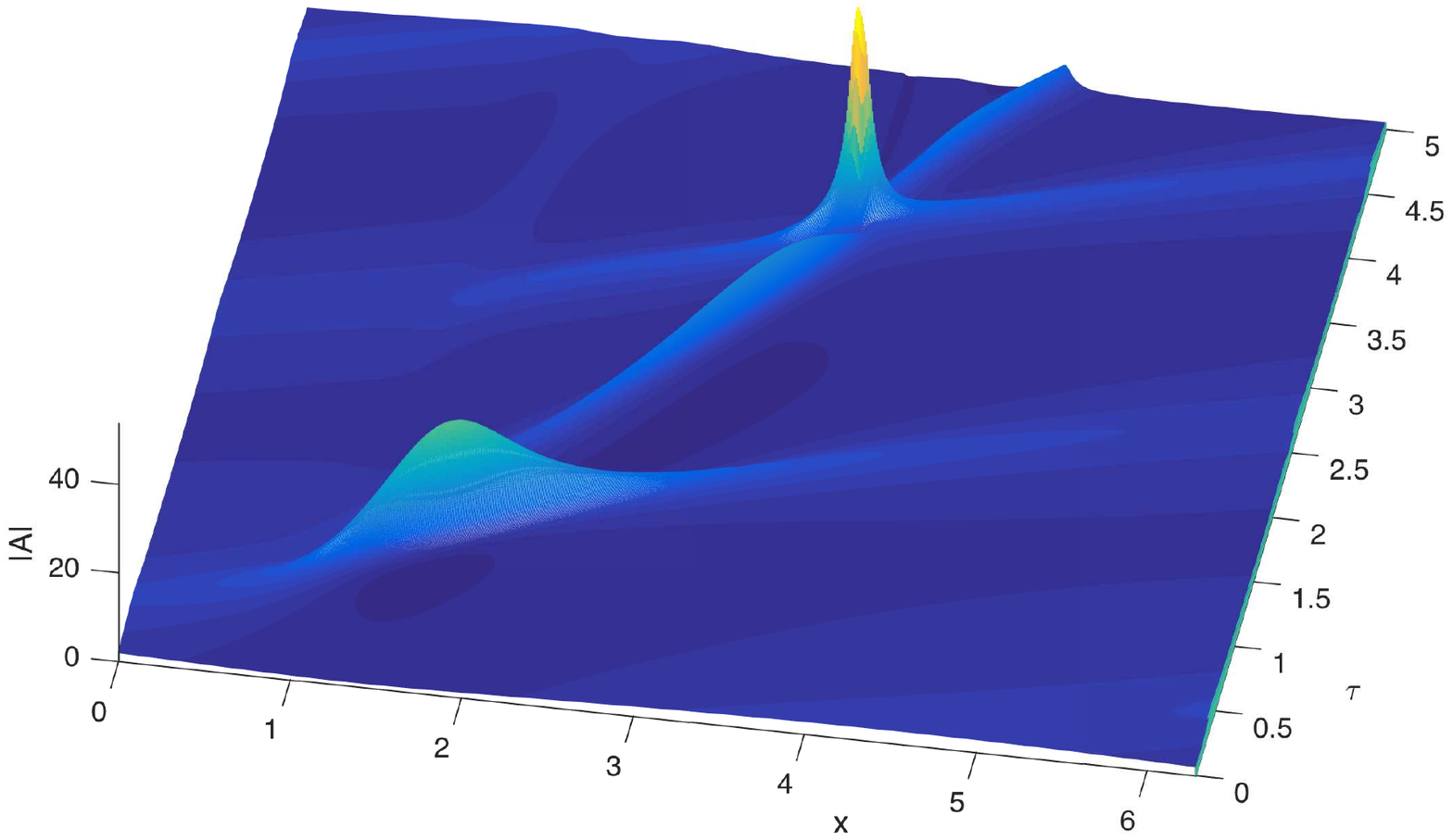}
\caption{A numerical solution of the unidirectional equation \eqref{u_eq} for $|A|=|u|$ with initial data \eqref{u_ic}.}
\label{fig:u}
\end{figure}

\begin{figure}
\includegraphics[width=4in]{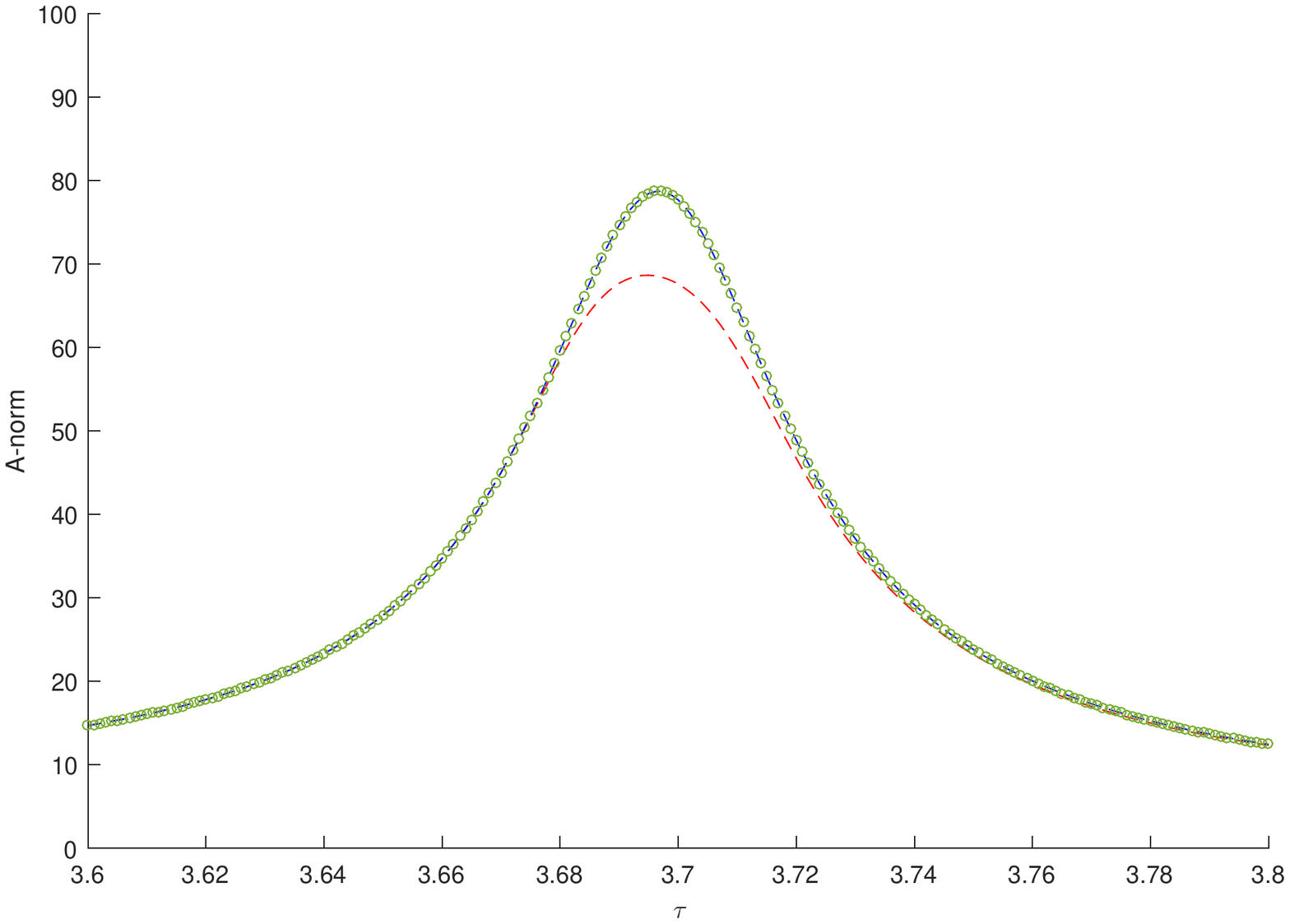}
\caption{Plot of the $A$-norm for the initial value problem shown in Figure~\ref{fig:u} with $N = 2^n$ Fourier modes, where $n=12$ (red dashed line), $n=14$ (blue dashed line), and $n=16$ (green circles). The solution appears to be fully converged for $N = 2^{14}$.}
\label{fig:u_a1}
\end{figure}

For comparison, we show a solution of the Szeg\"o equation with the same initial data in Figure~\ref{szego_eq}.
This solution does not exhibit the strong focusing of the previous ones.

\begin{figure}
\includegraphics[width=4in]{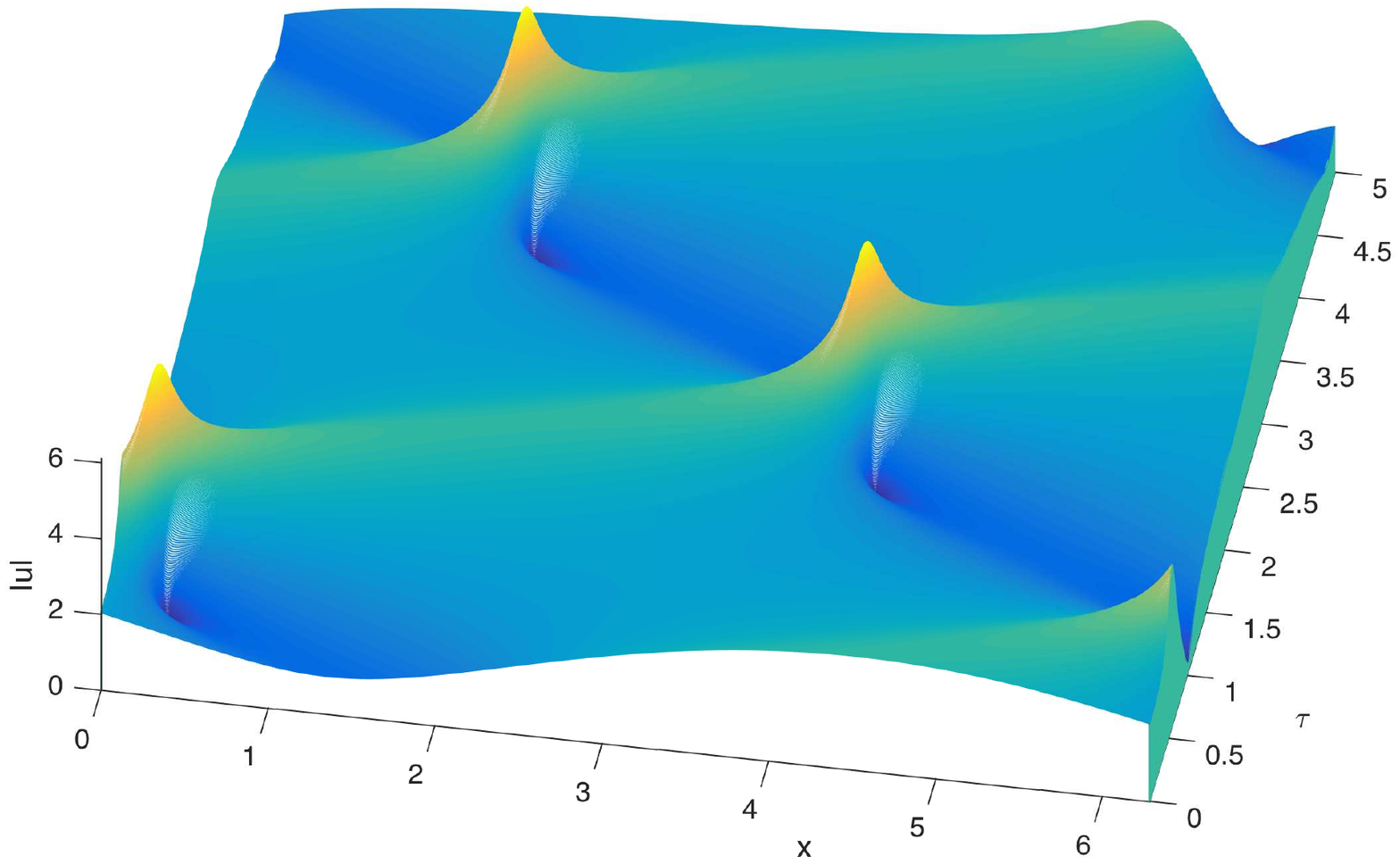}
\caption{A numerical solution of the Szeg\"o equation \eqref{szego_eq} for $|u|$ with initial data \eqref{u_ic}.}
\label{fig:szego}
\end{figure}

Finally, in Figures~\ref{fig:uv_disp_1}--\ref{fig:uv_disp_-1} we show two solutions of \eqref{uv_eq_disp} which illustrate the effect of positive and negative short-wave dispersion, respectively.

\begin{figure}
\includegraphics[width=4in]{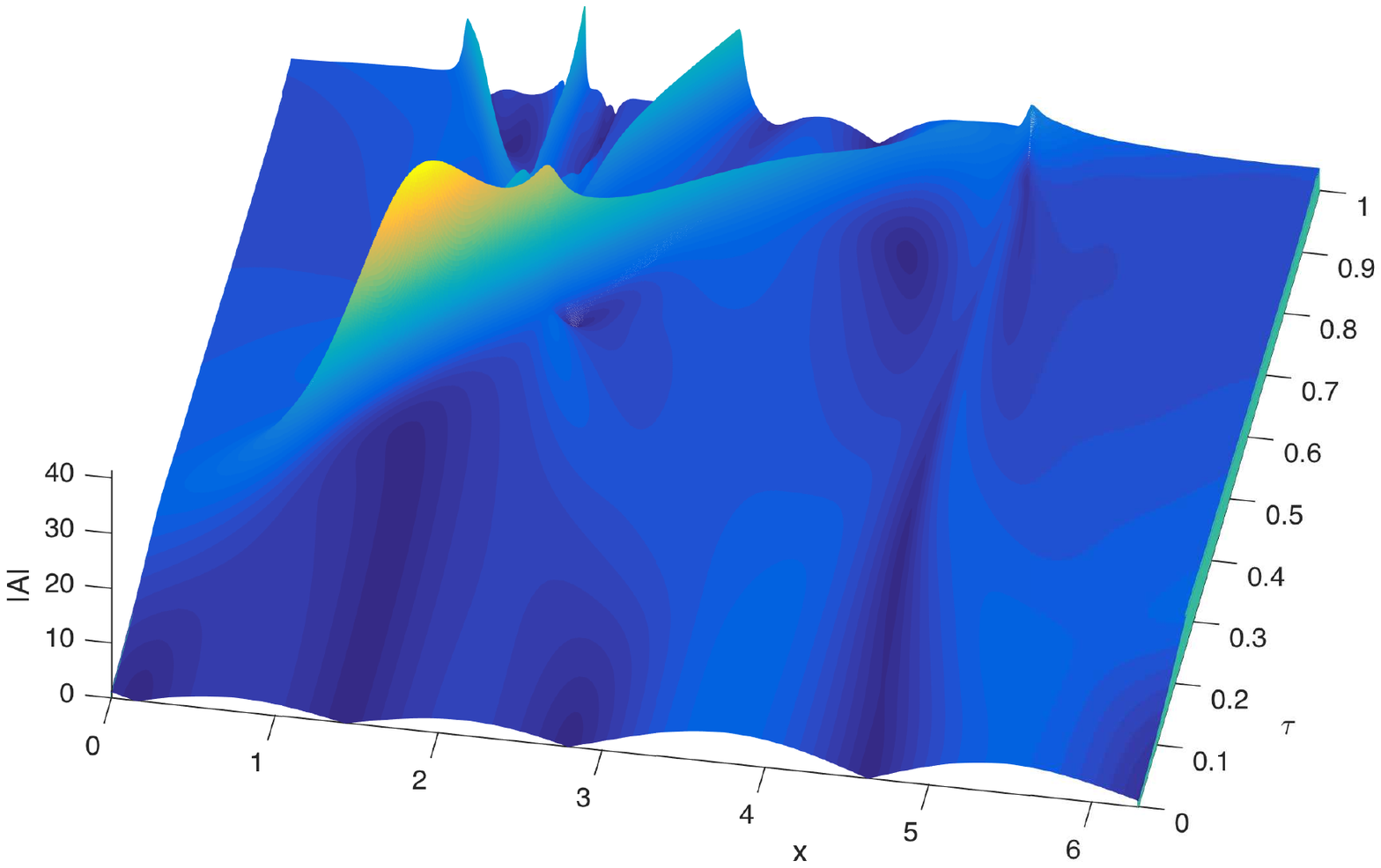}
\caption{Plot of $|A|$ for a numerical solution of \eqref{uv_eq_disp} with $\alpha=1$, $\beta=2$, $\gamma=0$, positive dispersion coefficient  $\nu=1$, and initial data \eqref{uv_ic}.}
\label{fig:uv_disp_1}
\end{figure}

\begin{figure}
\includegraphics[width=4in]{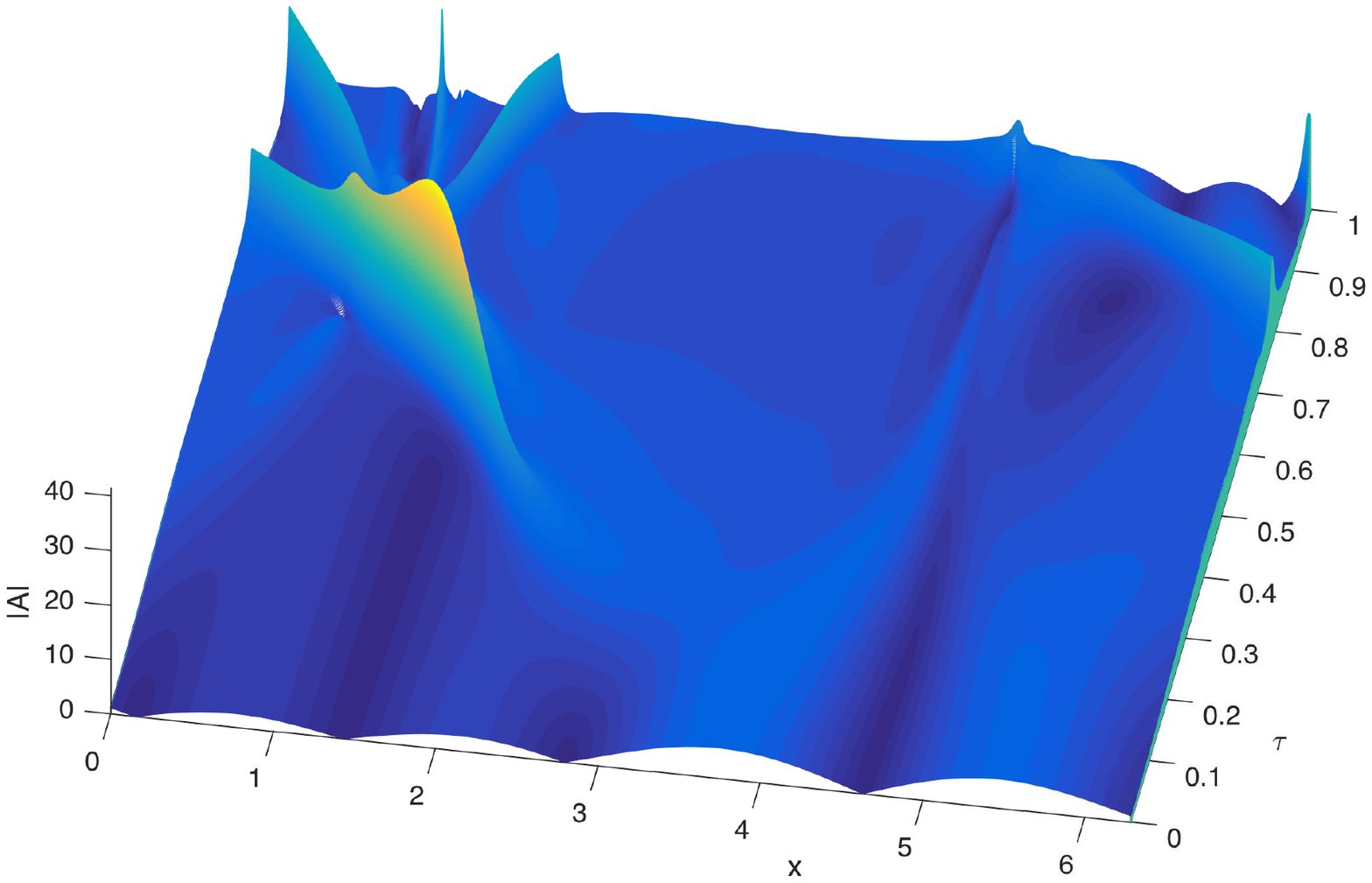}
\caption{Plot of $|A|$ for a numerical solution of \eqref{uv_eq_disp} with $\alpha=1$, $\beta=2$, $\gamma=0$, negative dispersion coefficient $\nu=-1$, and initial data \eqref{uv_ic}.}
\label{fig:uv_disp_-1}
\end{figure}

\section{The asymptotic expansion}
\label{sec:app}

In this section, we describe the asymptotic expansion in more detail.
We look for solutions of the dimensionless Maxwell equations \eqref{non_dim_max} that depend on $(x,z,t)$ and decay as $|z|\to\infty$.
We assume the constitutive relations \eqref{non_dim_con}, where the permittivity and susceptibility are given in $z>0$ and $z<0$
by \eqref{eps_pm} and \eqref{perm_exp}. The jump conditions across the interface $z=0$ are given by \eqref{jump_con}.

We introduce a `slow' time variable $\tau=\eps^2 t$ and expand time derivatives as
\[
\partial_t = \partial_t + \eps^2 \partial_\tau.
\]
We then expand the solutions as in \eqref{asy_exp} and equate coefficients of powers of $\eps$ in the result.

\subsection{Order $\eps$.}

At first-order, we find that the electric fields satisfy
\begin{align*}
\nabla \cdot \mathbf{D}_1  = 0, \qquad \nabla \times \mathbf{E}_1 = 0,
\end{align*}
with the jump conditions across $z=0$
\[
[\mathbf{n} \cdot \mathbf{D}_1 ] =  0,
\qquad
[\mathbf{n}  \times \mathbf{E}_1 ] = 0.
\]
The leading-order constitutive relation is
\[
\mathbf{D}_1(x,z,t,\tau) = \int \epsilonr(t-t') \mathbf{E}_1(x,z,t',\tau)\, dt',
\]
where the slow time variable $\tau$ occurs as a parameter. The expansion of the constitutive relation is derived below.

The electric field is the gradient of a harmonic potential in $z>0$ and $z<0$, and the Fourier-Laplace solutions that decay as $|z|\to\infty$ are
\[
\mathbf{E}_1 = \begin{cases} \hat{\mathbf{\Phi}}_1^+ e^{ikx - |k|z-i \omega t} + \text{c.c.} &\mbox{if $z>0$,}\\
\hat{\mathbf{\Phi}}_1^- e^{ikx + |k|z-i \omega t}  + \text{c.c.} &\mbox{if $z<0$,}\end{cases}
\qquad \hat{\mathbf{\Phi}}_1^{\pm} = A^\pm\left[\begin{array}{c}
            1 \\
              0  \\
          \pm i \sgn(k)
     \end{array} \right],
\]
where c.c.\ stands for the complex-conjugate of the preceding term. Moreover,
\[
\mathbf{D}_1 = \hepsilonr(\omega) \mathbf{E}_1,
\]
where $\hepsilonr = \hepsilonr^+$ in $z>0$ and $\hepsilonr = \hepsilonr^-$ in $z<0$.
The jump condition for $\mathbf{n}\times \mathbf{E}_1$ implies that $A^+=A^-$, and then the jump condition
for $\mathbf{n}\cdot\mathbf{D}_1$ implies that $\omega = \omega_0$ where
\begin{equation}
\hepsilonr^+(\omega_0) + \hepsilonr^-(\omega_0) = 0.
\label{eps0_disp}
\end{equation}
Fourier superposing these solutions over $k$, we get the linearized solution in \eqref{lin_sol},
\begin{align*}
\mathbf{E}_1 &= \mathbf{\Phi}_1(x,z,\tau) e^{-i\omega_0 t} + \text{c.c.},
\\
\mathbf{D}_1 &= \hepsilonr(\omega_0) \mathbf{\Phi}_1(x,z,\tau) e^{-i\omega_0t} + \text{c.c.}.
\end{align*}


\subsection{Order $\eps^2$.}

At second-order, we find that the magnetic fields satisfy
\begin{align*}
\nabla \cdot \mathbf{B}_2 = 0 , \qquad \nabla \times \mathbf{H}_2 =  \mathbf{D}_{1t}, \qquad \mathbf{B}_2 = \mu \mathbf{H}_2,
\end{align*}
with jump conditions $[\mathbf{n}\cdot \mathbf{B}_2] = [\mathbf{n}\times \mathbf{H}_2] = 0$. The solution is
\begin{align}
\begin{split}
&\mathbf{H}_2(x,z, \tau, t) = \mathbf{\Xi}_2(x,z, \tau) e^{-i \omega_0 t} + \text{c.c.},
\\
&\mathbf{\Xi}_2(x,z, \tau) = - i\sgn(z)\omega_0 \hepsilonr(\omega_0)  \left(\int\frac{1 }{ |k|  } \hat{A}(k, \tau)  e^{ikx -  |kz|} \, dk\right)
\left[\begin{array}{c}
        0 \\
        1 \\
        0
     \end{array} \right].
     \end{split}
     \label{h2}
\end{align}

\subsection{Expansion of the constitutive relation.}

The solution for the leading-order electric field has a frequency spectrum that is concentrated in a narrow band of order $\eps^2$ width about $\omega_0$.
We expand the linearized constitutive relation in  frequency space as
\begin{align*}
\hat{\mathbf{D}}(\omega) &= \hepsilonr(\omega) \hat{\mathbf{E}}(\omega)
\\
&=\left[\hepsilonr(\omega_0) + \hepsilonr'(\omega_0)( \omega - \omega_0) + O( \omega - \omega_0)^2 \right] \hat{\mathbf{E}}(\omega),
\end{align*}
where the prime denotes a derivative with respect to $\omega$  and we omit the spatial dependence.

If $\mathbf{E} = \mathbf{\Phi}(\tau) e^{-i\omega_0 t}$, then inversion the Fourier transform
gives
\[
\mathbf{D} = \hepsilonr(\omega_0)\mathbf{\Phi}(\tau) e^{-i\omega_0 t}  + i \eps^2 \hepsilonr'(\omega_0)
\mathbf{\Phi}_{\tau} e^{ -i \omega_0 t} +  O(\eps^4).
\]
Similarly, if $\mathbf{E} = \mathbf{\Phi}^*(\tau) e^{i\omega_0 t}$, then
\[
\mathbf{D} = \hepsilonr(\omega_0)\mathbf{\Phi}^*(\tau) e^{i\omega_0 t}  - i \eps^2 \hepsilonr'(\omega_0)
\mathbf{\Phi}^*_{\tau} e^{i \omega_0 t} +  O(\eps^4),
\]
where we have used the reality-condition that $\hepsilonr$ is an even function of $\omega$.
The nonlinear susceptibility $\hat{\chi}$ and the imaginary part $\hepsiloni$ of the permittivity
are evaluated at $\pm\omega_0$ to leading order in $\eps$.

The third-order electric field has the form
\begin{equation}
\mathbf{E}_3 = \mathbf{\Phi}_3(x,z,\tau) e^{-i\omega_0 t} + \text{c.c.} + \text{n.r.t.},
\label{e3}
\end{equation}
where n.r.t.\ stands for nonresonant terms proportional
to $e^{\pm 3i\omega_0 t}$, which do not effect the equation for $A$.

Using the expansion $\mathbf{E} = \eps \mathbf{E}_1 + \eps^3 \mathbf{E}_3 +O(\eps^5)$ and \eqref{perm_exp} in \eqref{non_dim_con}, we find that
\begin{align}
\mathbf{D}_3 &= \left[\hepsilonr(\omega_0)
\mathbf{\Phi}_3 + \mathbf{A}_1 + \mathbf{F}_1\right] e^{-i\omega_0 t}+\text{c.c.} + \text{n.r.t.},
\label{d3}
\end{align}
where, after the use of \eqref{sym_chi},
\begin{align}
\begin{split}
\mathbf{A}_1 &= i \hepsilonr'(\omega_0) \mathbf{\Phi}_{1\tau} + i\hepsiloni(\omega_0) \mathbf{\Phi}_1,
\\
\mathbf{F}_1 &= 2\hat{\chi}(\omega_0, -\omega_0, \omega_0) (\mathbf{\Phi}_1^*\cdot\mathbf{\Phi}_1) \mathbf{\Phi}_1
+ \hat{\chi}(\omega_0, \omega_0, -\omega_0) (\mathbf{\Phi}_1\cdot\mathbf{\Phi}_1) \mathbf{\Phi}_1^*.
\end{split}
\label{defAF}
\end{align}

\subsection{Order $\eps^3$.}
At third-order, we get equations for the electric fields
\begin{equation}
\nabla \cdot \mathbf{D}_3 = 0, \qquad \nabla \times \mathbf{E}_3 = - \mu\mathbf{H}_{2t},
\label{third_order}
\end{equation}
together with the jump conditions
\begin{align*}
  [ \mathbf{n} \cdot \mathbf{D}_3  ] = 0,\qquad
[ \mathbf{n} \times \mathbf{E}_3] = 0
\end{align*}
across $z=0$.
Using \eqref{h2}--\eqref{d3} in \eqref{third_order}, and the equation $\nabla\cdot \mathbf{\Phi}_1=0$,
we find that $\mathbf{\Phi}_3$ satisfies the nonhomogeneous PDEs
\begin{align*}
\nabla\cdot \mathbf{\Phi}_3 &=  - \frac{1}{\hepsilonr(\omega_0)}\nabla\cdot\mathbf{F}_1,
\\
\nabla\times \mathbf{\Phi}_3 &= i\mu\omega_0 \mathbf{\Xi}_2,
\end{align*}
in $z>0$ and $z<0$, with the jump conditions
\[
[\hepsilonr(\omega_0)
\mathbf{n}\cdot\mathbf{\Phi}_3] =  - [\mathbf{n}\cdot(\mathbf{A}_1+\mathbf{F}_1)],
\qquad [\mathbf{n}\times \mathbf{\Phi}_3] = 0.
\]
In addition, we require that
\[
\mathbf{\Phi}_3(x,z,\tau) \to 0\qquad\mbox{as $|z|\to\infty$}.
\]
The homogeneous form of these equations has a nontrivial solution, and the nohomogeneous terms must satisfy an appropriate solvability condition
if  a solution for $\mathbf{\Phi}_3$ is to exist.

To derive this solvability condition,
we write the equations in component form for
\begin{equation}
\mathbf{\Phi}_3 = \left(\begin{array}{c} E_1\\0\\ E_2\end{array}\right),
\quad
\mathbf{F}_1 = \left(\begin{array}{c} F_1\\0\\ F_2\end{array}\right),
\quad
\mathbf{A}_1 = \left(\begin{array}{c} A_1\\0\\ A_2\end{array}\right),
\quad
\mathbf{\Xi}_2 = \left(\begin{array}{c} 0\\D\\ 0\end{array}\right),
\label{def_cap}
\end{equation}
and take the Fourier transform with respect to $x$, where
\[
E_1(x,z,\tau) = \int \hat{E}_1(k,z,\tau) e^{ikx}\, dk,
\]
and similarly for the other variables.

After some algebra, and the use of \eqref{eps0_disp}, we find that the equations reduce to the ODE
\begin{equation}
-k^2\hat{E}_{1}  + \hat{E}_{1zz} =   \frac{ k^2\hat{F}_{1} - ik \hat{F}_{2z}}{\hepsilonr(\omega_0)} +  i \omega_0\mu  \hat{D}_z
\label{e1_eq}
\end{equation}
for $\hat{E}_1 = \hat{E}_1^\pm$ with $\hepsilonr = \hepsilonr^\pm$ in $z>0$ and $z<0$, the jump conditions
\begin{align}
 \hat{E}^+_{1z}  + \hat{E}^-_{1z} &= - ik\frac{( \hat{A}_2^+ - \hat{A}_2^-) + ( \hat{F}_2^+ - \hat{F}_2^-)}{\hepsilonr^+(\omega_0)}
 + i \omega_0 \mu  ( \hat{D}^+  +   \hat{D}^-),
 \label{jump1}
 \\
\hat{E}^+_{1} -  \hat{E}^-_{1}  &= 0,
\label{jump2}
\end{align}
where the functions in \eqref{jump1}--\eqref{jump2} are evaluated at $z=0$, and the decay conditions
\[
\hat{E}_1^+(k,z,\tau) \to 0\quad\mbox{as $z\to\infty$},\qquad \hat{E}_1^-(k,z,\tau) \to 0\quad\mbox{as $z\to-\infty$}.
\]

Integration by parts yields
\begin{align*}
&\int_{0}^{\infty} e^{-|k|z}( \hat{E}_{1zz} - k^2 \hat{E}_{1}) \ dz = -\hat{E}_{1z}^+  - |k|\hat{E}_{1}^+,
\\
&\int^{0}_{-\infty} e^{|k|z}( \hat{E}_{1zz} - k^2 \hat{E}_{1}) \ dz = \hat{E}_{1z}^-  - |k|\hat{E}_{1}^-,
\end{align*}
where the boundary terms on the right-hand side are evaluated at $z=0$, and the boundary terms at infinity vanish as a result of the decay conditions.
Subtracting these two equations, then using the second jump condition \eqref{jump2} and the differential equation \eqref{e1_eq}, we get that
\begin{align*}
 \hat{E}_{1z}^+ + \hat{E}_{1z}^- &= \int^{0}_{-\infty} e^{|k|z} \left(     \frac{ k^2\hat{F}_{1} - ik \hat{F}_{2z}}{\hepsilonr(\omega_0)}
 +  i \omega_0 \mu \hat{D}_z  \right) \, dz
 \\
 &\qquad- \int_{0}^{\infty} e^{-|k|z}   \left(  \frac{ k^2\hat{F}_{1} - ik \hat{F}_{2z}}{\hepsilonr(\omega_0)} +  i \omega_0 \mu  \hat{D}_z  \right) \, dz \\
  &\qquad  - \frac{ik}{\epsilonr^+(\omega_0)} \left(  \hat{F}^+_2   - \hat{F}^-_2   \right) + i \omega_0\mu  \left(  \hat{D}^+ +  \hat{D}^- \right).
\end{align*}
Using this equation to eliminate $\hat{E}_1$ from the first jump condition \eqref{jump1}, and simplifying the result, we find that
\begin{align}
\begin{split}
-\frac{ ik}{\hepsilonr^+(\omega_0)} \left.\left(\hat{A}_2^+\ - \hat{A}_2^-\right)\right|_{z=0}
 &= \int^{0}_{-\infty} e^{|k|z} \left(     \frac{ k^2\hat{F}_{1} - ik \hat{F}_{2z}}{\hepsilonr(\omega_0)}
 +  i \omega_0 \mu \hat{D}_z  \right) \, dz
 \\
 &- \int_{0}^{\infty} e^{-|k|z}   \left(  \frac{ k^2\hat{F}_{1} - ik \hat{F}_{2z}}{\hepsilonr(\omega_0)} +  i \omega_0 \mu  \hat{D}_z  \right) \, dz,
\end{split}
\label{solv_con}
\end{align}
which is the required solvability condition.

Finally, we use \eqref{def_cap}, \eqref{defAF}, \eqref{h2}, and \eqref{lin_sol} to express the functions in \eqref{solv_con} in terms of $\hat{A}$
and evaluate the resulting $z$-integrals on the right-hand side. After some algebra, which we omit, we get equation \eqref{spec_eq} for $\hat{A}$.

\section{Appendix: two lemmas}
\label{sec:lem}

For completeness, we prove two simple estimates. First, we prove the commutator estimate used in the proof of Theorem~\ref{th:short_time}.

\begin{lemma}
\label{lem:com_est}
If $f,g\in \dot{H}^s(\Ts)$ and $s > 1/2$, then
\[
\left\|[\Pr,\partial_x^{-1} f]\, \partial_x g\right\|_s \lesssim \|f\|_A \|g\|_s + \|f\|_s \|g\|_A,
\]
where $[\Pr,F] = \Pr F - F\Pr$ denotes the commutator of the projection $\Pr$ with multiplication by $F$.
\end{lemma}

\begin{proof}
By density, it suffices to prove the estimate for zero-mean, $C^\infty$-functions
\[
f(x) = \sum_{n\in\Ir^*} \hat{f}_n e^{inx},\qquad g(x) = \sum_{n\in\Ir^*} \hat{g}_n e^{inx}.
\]
We have
\[
|\partial_x|^s [\Pr,\partial_x^{-1} f] \partial_x g
= \sum_{m,n\in \Ir^*} \frac{n}{m} |m+n|^s \left\{\theta_{m+n} - \theta_n\right\} \hat{f}_m \hat{g}_n e^{i(m+n)x},
\]
where
\[
\theta_n = \begin{cases} 1 &\mbox{if $n>0$,} \\ 0 &\mbox{if $n < 0$.}\end{cases}
\]
If $\theta_{m+n} - \theta_n \ne 0$, then $m+n$ and $n$ have opposite signs, so $|n| \le |m|$ and therefore
\[
\left|\frac{n}{m}|m+n|^s \left\{\theta_{m+n} - \theta_n\right\}\right| \lesssim |m|^{s} + |n|^{s}.
\]
The result then follows from Parseval's theorem and Young's inequality.
\end{proof}

Next, we prove the estimate used in the proof of Theorem~\ref{th:global_weak}.

\begin{lemma}
\label{lem:w}
If $f,g \in L^2_+(\Ts)$ and $s > 3/2$, then $f g_x \in H_+^{-s}(\Ts)$ and
\[
\left\|f g_x\right\|_{H^{-s}} \lesssim \|f\|_{L^2} \|g\|_{L^2}.
\]
\end{lemma}

\begin{proof}
By density, it is sufficient to prove the estimate for smooth functions. Let
\[
\hat{f}(n) = \frac{1}{2\pi} \int_{\Ts} f(x) e^{-inx}\, dx,\qquad
\hat{g}(n) = \frac{1}{2\pi} \int_{\Ts} g(x) e^{-inx}\, dx
\]
denote the Fourier coefficients of $f$ and $g$, which vanish for $n\le 0$. Then
the $n$th Fourier coefficient of $h=fg_x$ is zero for $n\le 1$, and for $n \ge 2$ it is given by
the finite sum
\[
\hat{h}(n) = \sum_{k=1}^{n-1} \hat{f}(n-k)\cdot ik \hat{g}(k).
\]
It follows that
\[
|\hat{h}(n)| \le n \|\hat{f}*\hat{g}\|_{\ell^\infty} \le n \|\hat{f}\|_{\ell^2} \|\hat{g}\|_{\ell^2},
\]
so if $s > 3/2$, we get that
\begin{align*}
\|h\|_{H^{-s}} &= \left(\sum_{n=2}^\infty \frac{|\hat{h}(n)|^2}{n^{2s}}\right)^{1/2}
\\
&\le \left(\sum_{n=2}^{\infty}\frac{1}{n^{2s-2}}\right)^{1/2} \|\hat{f}\|_{\ell^2} \|\hat{g}\|_{\ell^2}
\\
&\lesssim \|f\|_{L^2} \|g\|_{L^2}.
\end{align*}
\end{proof}


\begin{thebibliography}{10}

\bibitem{biello} J.~Biello and J.~K.~Hunter, Nonlinear Hamiltonian waves with constant frequency and
surface waves on vorticity discontinuities, \emph{Comm. Pure Appl. Math.} \textbf{63}, 2009, 303--336.

\bibitem{helicity} K.~Y.~Bliokh and F.~Nori, Transverse spin of a surface polariton, \emph{Phys.~Rev.~A}, \textbf{85}, 061801(R) (2012)



\bibitem{boyd} R.~W.~Boyd, \emph{Nonlinear Optics}, Third Ed., Elsevier, 2008.

\bibitem{boyer} F.~Boyer and P.~Fabrie, \emph{Mathematical Tools for the Study of the Incompressible Navier-Stokes Equations and Related Models},
Springer, 2013.


\bibitem{brezis} H.~Brezis, \emph{Functional Analysis, Sobolev Spaces and Partial Differential Equations}, Springer, 2011.




\bibitem{morrison} C.~Chandre, L.~de Guillebon, A.~Back, E.~Tassi, and P.~J.~Morrison,
On the use of projectors for Hamiltonian systems and their relationship with Dirac brackets,
\emph{J. Phys. A: Math. Theor.}, \textbf{46} (2013).


\bibitem{szego} P.~G\'erard and S.~Grellier,  The Szeg\"o cubic equation, \emph{Ann. Sci. \'Ec. Norm. Sup\'er.}, \textbf{43} (2010), 761--810.



\bibitem{zayats} M.~Kauranen and A.~V.~Zayays, Nonlinear plasmonics, \emph{Nature Photonics}, \textbf{6}, 737--748, (2012).

\bibitem{plasmonics} A.~A.~Maradudin, J.~R.~Sambles, and W.~L.~Barnes, ed., \emph{Modern Plasmonics},  Handbook of Surface Science, Vol.~4, Elsevier, 2014.


\bibitem{newell} J.~Moloney and  A.~Newell, \emph{Nonlinear Optics}, Westview Press, 2003.

\bibitem{west} P.~R.~West, S. Ishii, G.~Naik, N.~Emani, and V.M.~Shalaev, Searching for better plasmonic materials, \emph{Laser and Photonics Review}, \textbf{2}, 795--808 (2010).

\bibitem{zakharov} V.~E.~Zakharov, Generalized Hamiltonian formalisn in nonlinear optics, in \emph{Soliton-driven Photonics}, ed. A.~D.~Boardman and A.~P.~Sukhorukov, Nato Science Series II, Vol. 31, 505--518, 2001.



\end{thebibliography}
\end{document}